\documentclass[11pt, twoside]{amsart}
\usepackage[english]{babel}
\theoremstyle{plain}

\usepackage{amsmath,amsfonts,amsthm,amssymb}
\usepackage{graphicx}
\usepackage{verbatim}
\usepackage{color}
\usepackage{amscd}
\usepackage{verbatim}

\addtocounter{page}{1}%
\newtheorem{theorem}{Theorem}[section]

\newtheorem{lemma}[theorem]{Lemma}

\newtheorem{definition}[theorem]{Definition}
\newtheorem{remark}[theorem]{Remark}
\newtheorem{proposition}[theorem]{Proposition}

\newtheorem{example}[theorem]{Example}

\newtheorem{Lem}[theorem]{Lemma}
\newtheorem{Pro}[theorem]{Proposition}

\newtheorem{Rem}[theorem]{Remark}

\newtheorem{Cor}[theorem]{Corollary}

\newtheorem*{corollary*}{Corollary}
\newtheorem*{theorem*}{Theorem}

\def \R {{\mathbb {R}}}
\def \N {{\mathbb N}}

\begin{document}

\setcounter{page}{1}

\newcommand{\M}{{\mathcal M}}
\newcommand{\loc}{{\mathrm{loc}}}
\newcommand{\core}{C_0^{\infty}(\Omega)}
\newcommand{\sob}{W^{1,p}(\Omega)}
\newcommand{\sobloc}{W^{1,p}_{\mathrm{loc}}(\Omega)}
\newcommand{\merhav}{{\mathcal D}^{1,p}}
\newcommand{\be}{\begin{equation}}
\newcommand{\ee}{\end{equation}}
\newcommand{\mysection}[1]{\section{#1}\setcounter{equation}{0}}
\newcommand{\laplace}{\Delta}
\newcommand{\pl}{\laplace_p}
\newcommand{\grad}{\nabla}
\newcommand{\pd}{\partial}
\newcommand{\bo}{\pd}
\newcommand{\csub}{\subset \subset}
\newcommand{\sm}{\setminus}
\newcommand{\ssm}{:}
\newcommand{\diver}{\mathrm{div}\,}
\newcommand{\bea}{\begin{eqnarray}}
\newcommand{\eea}{\end{eqnarray}}
\newcommand{\bean}{\begin{eqnarray*}}
\newcommand{\eean}{\end{eqnarray*}}
\newcommand{\thkl}{\rule[-.5mm]{.3mm}{3mm}}
\newcommand{\cw}{\stackrel{\rightharpoonup}{\rightharpoonup}}
\newcommand{\id}{\operatorname{id}}
\newcommand{\supp}{\operatorname{supp}}
\newcommand{\wlim}{\mbox{ w-lim }}
\newcommand{\mymu}{{x_N^{-p_*}}}
\newcommand{\abs}[1]{\lvert#1\rvert}
\newcommand{\pf}{\noindent \mbox{{\bf Proof}: }}


\renewcommand{\theequation}{\thesection.\arabic{equation}}
\catcode`@=11 \@addtoreset{equation}{section} \catcode`@=12
\newcommand{\Real}{\mathbb{R}}
\newcommand{\real}{\mathbb{R}}
\newcommand{\Nat}{\mathbb{N}}
\newcommand{\ZZ}{\mathbb{Z}}
\newcommand{\CC}{\mathbb{C}}
\newcommand{\Pess}{\opname{Pess}}
\newcommand{\Proof}{\mbox{\noindent {\bf Proof} \hspace{2mm}}}
\newcommand{\mbinom}[2]{\left (\!\!{\renewcommand{\arraystretch}{0.5}
\mbox{$\begin{array}[c]{c}  #1\\ #2  \end{array}$}}\!\! \right )}
\newcommand{\brang}[1]{\langle #1 \rangle}
\newcommand{\vstrut}[1]{\rule{0mm}{#1mm}}
\newcommand{\rec}[1]{\frac{1}{#1}}
\newcommand{\set}[1]{\{#1\}}
\newcommand{\dist}[2]{$\mbox{\rm dist}\,(#1,#2)$}
\newcommand{\opname}[1]{\mbox{\rm #1}\,}
\newcommand{\mb}[1]{\;\mbox{ #1 }\;}
\newcommand{\undersym}[2]
 {{\renewcommand{\arraystretch}{0.5}  \mbox{$\begin{array}[t]{c}
 #1\\ #2  \end{array}$}}}
\newlength{\wex}  \newlength{\hex}
\newcommand{\understack}[3]{%
 \settowidth{\wex}{\mbox{$#3$}} \settoheight{\hex}{\mbox{$#1$}}
 \hspace{\wex}  \raisebox{-1.2\hex}{\makebox[-\wex][c]{$#2$}}
 \makebox[\wex][c]{$#1$}   }%
\newcommand{\smit}[1]{\mbox{\small \it #1}}
\newcommand{\lgit}[1]{\mbox{\large \it #1}}
\newcommand{\scts}[1]{\scriptstyle #1}
\newcommand{\scss}[1]{\scriptscriptstyle #1}
\newcommand{\txts}[1]{\textstyle #1}
\newcommand{\dsps}[1]{\displaystyle #1}
\newcommand{\dx}{\,\mathrm{d}x}
\newcommand{\dy}{\,\mathrm{d}y}
\newcommand{\dz}{\,\mathrm{d}z}
\newcommand{\dt}{\,\mathrm{d}t}
\newcommand{\dr}{\,\mathrm{d}r}
\newcommand{\du}{\,\mathrm{d}u}
\newcommand{\dv}{\,\mathrm{d}v}
\newcommand{\dV}{\,\mathrm{d}V}
\newcommand{\ds}{\,\mathrm{d}s}
\newcommand{\dS}{\,\mathrm{d}S}
\newcommand{\dk}{\,\mathrm{d}k}

\newcommand{\dphi}{\,\mathrm{d}\phi}
\newcommand{\dtau}{\,\mathrm{d}\tau}
\newcommand{\dxi}{\,\mathrm{d}\xi}
\newcommand{\deta}{\,\mathrm{d}\eta}
\newcommand{\dsigma}{\,\mathrm{d}\sigma}
\newcommand{\dtheta}{\,\mathrm{d}\theta}
\newcommand{\dnu}{\,\mathrm{d}\nu}

\def\ga{\alpha}     \def\gb{\beta}       \def\gg{\gamma}
\def\gc{\chi}       \def\gd{\delta}      \def\ge{\epsilon}
\def\gth{\theta}                         \def\vge{\varepsilon}
\def\gf{\phi}       \def\vgf{\varphi}    \def\gh{\eta}
\def\gi{\iota}      \def\gk{\kappa}      \def\gl{\lambda}
\def\gm{\mu}        \def\gn{\nu}         \def\gp{\pi}
\def\vgp{\varpi}    \def\gr{\rho}        \def\vgr{\varrho}
\def\gs{\sigma}     \def\vgs{\varsigma}  \def\gt{\tau}
\def\gu{\upsilon}   \def\gv{\vartheta}   \def\gw{\omega}
\def\gx{\xi}        \def\gy{\psi}        \def\gz{\zeta}
\def\Gg{\Gamma}     \def\Gd{\Delta}      \def\Gf{\Phi}
\def\Gth{\Theta}
\def\Gl{\Lambda}    \def\Gs{\Sigma}      \def\Gp{\Pi}
\def\Gw{\Omega}     \def\Gx{\Xi}         \def\Gy{\Psi}

\newcommand{\Rn}{\mathbb{R}^n}
 \newcommand{\lanbox}{{\, \vrule height 0.25cm width 0.25cm depth 0.01cm \,}}
  \renewcommand{\qedsymbol}{\lanbox}
   \newcommand{\dd}{\mathrm{d}}

\renewcommand{\div}{\mathrm{div}}
\newcommand{\red}[1]{{\color{red} #1}}

\pagestyle{headings}
\title[Optimal Hardy inequalities in cones]{Optimal Hardy inequalities in cones}
\author{Baptiste Devyver}
\address{Baptiste Devyver, Department of Mathematics,  University of British Columbia, Vancouver, Canada}
\email{devyver@math.ubc.ca}
\author{Yehuda Pinchover}
\address{Yehuda Pinchover,
Department of Mathematics, Technion - Israel Institute of
Technology,   Haifa 32000, Israel}
\email{pincho@techunix.technion.ac.il}
\author{Georgios Psaradakis}
\address{Georgios Psaradakis, Department of Mathematics, Technion - Israel Institute of
Technology,   Haifa 32000, Israel}
\email{georgios@techunix.technion.ac.il}
\begin{abstract}
Let $\Omega$ be an open connected cone in $\mathbb{R}^n$ with vertex at the origin. Assume that the operator
$$P_\mu:=-\Delta-\frac{\mu}{\delta_\Omega^2(x)}$$ is {\em subcritical} in $\Omega$, where $\delta_\Omega$ is the distance function to the boundary of $\Omega$ and $\mu \leq 1/4$. We show that under some smoothness assumption on $\Omega$,  the following improved Hardy-type inequality
\begin{equation*}
 \int_{\Omega}|\nabla \varphi|^2\dx -\mu\int_{\Omega} \frac{|\varphi|^2}{\delta_\Omega^2}\dx \geq \lambda(\mu)\int_{\Omega}  \frac{|\varphi|^2}{|x|^2}\dx \qquad \forall \varphi\in C_0^\infty(\Omega),
\end{equation*}
holds true, and the Hardy-weight $\lambda(\mu)|x|^{-2}$ is optimal in a certain definite sense. The constant $\lambda(\mu)>0$ is given explicitly.

\medskip 


\end{abstract}
\maketitle
\section{Introduction}\label{sec_Int}
Let $P$ be a symmetric second-order linear elliptic operator with real coefficients, defined in a domain $\Gw$ of $\mathbb{R}^n$, and denote by $q$ its associated quadratic form. Suppose that $q(\vgf)\geq0$ for all $\vgf\in C_0^\infty(\Omega),$ i.e. $P$ is {\em nonnegative} ($P\geq0$) in $\Gw$. Then $P$ is called {\em subcritical} in $\Gw$ if there exists a nontrivial, nonnegative weight $W$ such that the following Hardy-type inequality holds true
\begin{equation}
\label{HardyType}
q(\vgf) \geq \lambda \int_\Omega W(x) |\varphi(x)|^2 \dx \qquad \forall\varphi \in C_0^\infty(\Omega),
\end{equation}
where $\lambda  > 0$ is a constant. If $P\geq 0$ in $\Gw$ and (\ref{HardyType}) is not true for any $W\gneqq0$, then $P$ is called {\it critical} in $\Gw$.

Given a subcritical operator $P$  in $\Omega$, there is a huge convex set of weights $W\gneqq 0$ satisfying \eqref{HardyType}.  A natural question is to find a weight function $W$ which is ``as large as possible" and satisfies \eqref{HardyType} (see Agmon \cite[Page 6]{Ag82}).

In the paper \cite{DFP}, the authors have constructed a Hardy-weight $W$, for a subcritical operator $P$, which is {\em optimal}
in a certain definite sense. For symmetric operators the main result of \cite{DFP} reads as follows.

\begin{theorem}[{\cite[Theorem~2.2]{DFP}}]\label{temp_def}
Assume that $P$ is subcritical in $\Gw$. Fix  a reference point $x_0\in \Gw$, and set $\Gw^\star:=\Gw\setminus \{x_0\}$. There exists a nonzero nonnegative weight $W$ satisfying the following properties:
\begin{enumerate}
\item[(a)] Denote by $\lambda_0=\lambda_0(P,W,\Gw^\star)$ the largest constant $\gl$ satisfying
 \begin{equation}
\label{HardyType_def}
q(\vgf) \geq \lambda \int_{\Gw^\star} W(x) |\varphi(x)|^2 \dx \qquad \forall\varphi \in C_0^\infty(\Omega^\star).
\end{equation}
Then $\gl_0>0$ and the operator   $P-\lambda_0 W$ is \textit{critical} in $\Gw^\star$; that is, the inequality
$$q(\vgf) \geq  \int_{\Gw^\star} V(x)|\vgf(x)|^2 \dx \qquad \forall \vgf \in C_0^\infty(\Omega^\star)$$
is not valid for any $V \gneqq \lambda_0 W$.

\item[(b)] The constant $\lambda_0$ is also the best constant for \eqref{HardyType_def} with test functions supported in $\Gw'\subset \Gw$, where $\Gw'$ is either the complement of any fixed compact set in $\Omega$ containing $x_0$ or any fixed punctured neighborhood of $x_0$.
\item[(c)] The operator $P - \lambda_0 W$ is {\it null-critical} in $\Gw^\star$; that is, the corresponding Rayleigh-Ritz variational problem
\begin{equation}\label{RRVP}
\inf_{\vgf\in \mathcal{D}_P^{1,2}(\Gw^\star)}\frac{q(\varphi)}{\int_{\Gw^\star} W(x) |\vgf(x)|^2 \dx}
\end{equation}
 admits no minimizer. Here $\mathcal{D}_P^{1,2}(\Omega^\star)$ is the completion of $C_0^\infty(\Omega^\star)$ with respect to the norm $u\mapsto \sqrt{q(u)}$.
\item[(d)] If furthermore $W>0$ in $\Gw^\star$, then the spectrum and the essential spectrum of the Friedrichs extension of the operator $W^{-1}P$ on $L^2(\Gw^\star, W\dx)$ are both equal to  $[\lambda_0,\infty)$.
\end{enumerate}
\end{theorem}
\begin{definition}\label{def_opt}{\em
A weight function that satisfies properties (a)--(d) is called an {\em optimal Hardy weight} for the operator $P$ in $\Gw$.
 }
\end{definition}
\noindent For related spectral results concerning optimal Hardy inequalities see \cite{D}.

\medskip

One may look at a punctured domain $\Gw^\star$ as a noncompact manifold with two ends $\bar{\infty}$ and $x_0$, where $\bar{\infty}$ denotes the ideal point in the one-point compactification of $\Gw$. In fact,  the results of Theorem~\ref{temp_def} are valid on such manifolds.  In \cite[Theorem~11.6]{DFP}, the authors extend Theorem~\ref{temp_def} and get an optimal Hardy-weight $W$ in the {\em entire} domain $\Gw$, in the case of {\em boundary singularities}, where the two singular points of the Hardy-weight are located at $\pd\Omega\cup \{\bar{\infty}\}$ and not at $\bar{\infty}$ and at an isolated interior point of $\Omega$ as in Theorem~\ref{temp_def}. The result reads as follows.
\begin{theorem}[{\cite[Theorem~11.6]{DFP}}]\label{thm_bs}
Assume that $P$ is subcritical in $\Gw$.  Suppose that the Martin boundary $\delta\Omega$ of the operator $P$ in $\Gw$ is equal to the minimal Martin boundary and is equal to $\partial\Omega\cup \{\xi_0,\xi_1\}$, where $\partial\Omega\sm \{\xi_0,\xi_1\}$ is assumed to be a regular manifold of dimension $(n-1)$ without boundary, and the coefficients of $P$ are locally regular up to $\partial\Omega\sm \{\xi_0,\xi_1\}$.

Denote by $\hat{\Omega}$ the Martin compactification of $\Omega$, and assume that there exists a bounded domain $D\subset \Gw$ such that $\xi_0$ and $\xi_1$ belong to two different connected components $D_0$ and $D_1$ of $\hat{\Omega}\setminus \bar{D}$ such that each $D_j$ is a neighborhood in $\hat{\Omega}$ of $\xi_j$, where $j=0,1$.

Let $u_0$ and $u_1$ be the minimal Martin functions at $\xi_0$ and $\xi_1$ respectively. Consider the supersolution
$u_{1/2}:= (u_0u_1)^{1/2}$ of the equation $Pu=0$ in $\Gw$, and assume that
\begin{equation}\label{u1u0bs}
\lim_{\substack{x\to \gz_0\\x\in \Gw}} \frac{u_1(x)}{u_0(x)}=\lim_{\substack{x\to \gz_1\\ x\in \Gw}} \frac{u_0(x)}{u_1(x)}= 0.
\end{equation}
\vskip 3mm Then the weight $W:=\frac{Pu_{1/2}}{u_{1/2}}$ is an optimal Hardy weight for $P$ in $\Gw$. Moreover, if $W$ does not vanish on $\hat{\Omega}\setminus\{\xi_0,\xi_1\}$, then the spectrum and the essential spectrum  of the Friedrichs extension of the operator $W^{-1}P$ acting on $L^2(\Omega,W\mathrm{d}x)$ is $[1,\infty)$.
\end{theorem}
The following example illustrates Theorem~\ref{thm_bs} and motivated our present study.
\begin{example}[{\cite[Example~11.1]{DFP}}]\label{ex1}
{\em
Let $P=P_0:=-\Gd$, and consider the cone $\Gw$ with vertex at the origin, given by
\begin{equation}\label{def_cone}
    \Gw:=\{x\in \mathbb{R}^n\mid  r(x)>0,\gw(x)\in \Gs\}\,,
\end{equation}
where $\Gs$ is a Lipschitz domain on the unit sphere $\mathbb{S}^{n-1}\subset \mathbb{R}^n$, $n\geq 2$, and $(r,\gw)$ denotes the spherical coordinates of $x$ (i.e., $r=|x|$, and $\gw=x/|x|$). We assume that $P$ is subcritical in $\Gw$.

Let $\gf$ be the principal eigenfunction of the (Dirichlet) Laplace-Beltrami operator $-\Gd_S$ on $\Gs$ with principal eigenvalue $\gs=\gl_0(-\Gd_S,\mathbf{1},\Gs)$ (for the definition of $\gl_0$ see \eqref{eq_gl_0}), and set $$\gg_\pm:= \frac{2-n \pm \sqrt{(2-n)^2+4\gs}}{2}\,.$$ Then the positive harmonic functions $$u_\pm(r,\gw):=r^{\gg_\pm}\gf(\gw)$$ are the Martin kernels at $\infty$ and $0$ \cite{P94} (see also \cite{An12}).

The function
$$u_{1/2}:=(u_+u_-)^{1/2}=r^{(2-n)/2}\gf(\gw)$$ is a supersolution of the equation $Pu=0$ in $\Gw$ (this is the so called {\em supersolution construction for $P$ in $\Gw$ with the pair $(u_+,u_-)$}).

Consequently, the associated Hardy weight is
$$W(x):=\frac{Pu_{1/2}}{u_{1/2}}=\frac{(n-2)^2+4\gs}{4|x|^2},\, $$ and the corresponding Hardy-type inequality reads as follows
\begin{equation}\label{eq_4n2}
\int_{\Gw}|\nabla \varphi|^2\dx\geq \frac{(n-2)^2+4\gs}{4}\int_{\Gw}\frac{|\varphi|^2}{|x|^2}\dx \qquad \forall\varphi\in C_0^\infty(\Gw).
\end{equation}
It follows from  Theorem~\ref{thm_bs} that $W$ is an optimal Hardy-weight. Note that for $\Gs=\mathbb{S}^{n-1}$ we obtain the classical Hardy inequality in the punctured space. We also remark that the Hardy-type inequality \eqref{eq_4n2} and the {\em global} optimality of the constant
$(n-2)^2/4+\gs$ are not new (cf. \cite{Nz,LLM}).
 }
\end{example}
\noindent Let $$\gd(x)=\gd_\Gw(x):=\mathrm{dist}\,(x,\partial \Gw)$$ be the distance function to the boundary of a domain $\Gw$.

The aim of the present paper is to extend the result in Example~\ref{ex1} to the case of the Hardy operator $$P_\mu:=-\Gd-\frac{\gm}{\gd_\Gw^2(x)} \qquad \mbox{in } \Gw,$$ where $\Gw$ is the cone defined by \eqref{def_cone}, and $\gm\leq \gm_0:=\gl_0(-\Gd,\gd_\Gw^{-2},\Gw)$  under the assumption the $P_\gm$ is subcritical in $\Gw$ (for the definition of $\gl_0$, see \eqref{eq_gl_0}). In particular, we obtain an explicit expressions for the optimal Hardy weight $W$ corresponding to the singular points $0$ and $\infty$, for the associate best Hardy constant, and for the corresponding ground state. Note that since the potential $\gd_\Gw^{-2}(x)$ is singular on $\pd \Gw$, Theorem~\ref{thm_bs} is not applicable for $P_\gm$ with $\gm\neq 0$,  and we had to come up with new techniques and ideas to treat this case. For some recent results concerning sharp Hardy inequalities with boundary singularities see \cite{C,FM,GR} and references therein.

\medskip

The outline of the present paper is as follows. In Section~\ref{sec_pre} we fix the setting and notations, and introduce some basic definitions. In Section~\ref{sec_multi} we use an approximation argument to obtain two positive multiplicative solutions of the equation $P_\gm u=0$ in $\Gw$ of the form $u_\pm(r,w):=r^{\gg_\pm}\gth(\gw)$, while in Section~\ref{nsec_multi} we use the boundary Harnack principle of A.~Ancona~\cite{An87} and the methods in \cite{LP,P94} to get an explicit representation theorem for the positive solutions of the equation $P_\gm u=0$ in $\Gw$ that vanish  (in the potential theory sense) on $\partial\Gw\setminus \{0\}$. The obtained two linearly independent positive multiplicative solutions are the building blocks of the supersolution construction that is used in Section~\ref{sec_main} to prove our main result. In Section~\ref{sec_FTT} we consider a family of Hardy inequalities in the half-space $\R^n_+$ obtained by S. Filippas, A. Tertikas and J. Tidblom \cite{FTT}, and we obtain, for the appropriate case, the optimality of the corresponding weight.

We conclude the paper in Section~\ref{sec_DE} by proving
a closely related  Hardy-type inequality with the best constant for the (nonnegative) operator $P_\gm$ in $\Gw$, where $\Gw$ is a domain in $\R^n$ such that $0\in\partial \Gw$, and $\gd_\Gw$ satisfies (in the weak sense) the linear differential inequality
\begin{equation}
-\Gd \gd_\Gw +\frac{n-1+\sqrt{1-4\mu}}{|x|^2}\big(x\cdot\nabla \gd_\Gw-\gd_\Gw\big) \geq 0\qquad \mbox{in }\Gw.
\end{equation}
Finally, we note that parts of the results of the present paper were announced in \cite{DPP}.
\section{Preliminaries}\label{sec_pre}
In this section we fix our setting and notations, and introduce some basic definitions. We denote $\R_+:=(0,\infty)$, and
$$\Real^n_+:= \{(x_1,x_2,\ldots, x_n)\in \R^n \mid x_1>0\}.$$
Throughout the paper $\Omega$ is a domain in $\Real^{n}$, where $n\geq 2$. The distance function to the boundary of $\Gw$ is denoted by $\gd_{\Gw}$.
We write $\Gw' \Subset \Omega$ if $\Omega$ is open, $\overline{\Gw'}$ is
compact and $\overline{\Gw'} \subset \Omega$. By an {\em exhaustion} of $\Omega$ we mean a sequence $\left\{\Omega_{k}\right\}$ of smooth, relatively compact domains such that $x_{0} \in \Omega_{1}$, $\Omega_{k} \Subset \Omega_{k+1}$, and $\bigcup_{N=1}^{\infty}\Omega_{k}=\Omega$.

Let $f,g:\Gw \to [0,\infty)$. We denote $f \asymp g$ in $\Gw$ if there exists a positive constant $C$ such that $C^{-1}g \leq f \leq Cg$ in $\Gw$. Also, we write $f\gneqq 0$ in $\Gw$ if $f\geq 0$ in $\Gw$ but $f\neq 0$ in $\Gw$. We denote by $\mathbf{1}$ the constant function taking the value $1$ in $\Gw$.  $B_r(x)$ is the open ball of radius $r$ centered at $x$. If $\Gw$ is a cone and $R>0$, we denote by $A_R$ the annulus
$$A_R:=\{z\in \Gw\mid \frac{R}{2}\leq |z|\leq 2R\}.$$

In the present paper we consider a second-order linear elliptic operator $P$ defined on a domain $\Gw\subset \R^n$, and let $W\gneqq 0$ be a given function. We write $P\geq 0$ in $\Gw$ if the equation $Pu=0$ in $\Gw$ admits a positive (super)solution. Unless otherwise stated it is assumed that $P\geq 0$ in $\Gw$.

Throughout the paper it is assumed that the operator $P$ is {\em symmetric} and locally uniformly elliptic.  Moreover, we assume that coefficients of $P$ and the function $W$ are real valued and locally sufficiently regular in $\Omega$ (see \cite{DFP}). For such an operator $P$, potential $W$,  and $\gl\in\R$, we denote $P_\gl:=P-\gl W$.

The following well known Agmon-Allegretto-Piepenbrink (AAP) theorem holds (see for example \cite{Ag83} and references therein).
\begin{theorem}[The AAP Theorem]\label{thm_AAP}
Suppose that $P$ is symmetric, and let $q$ be the corresponding quadratic form.  Then
$P\geq 0$ in $\Gw$ if and only if $q(\varphi)\geq 0$ for every $\varphi\in C_0^\infty(\Omega)$.
\end{theorem}
We recall the following definitions.
\begin{definition}{\em  Let $q$ be the quadratic form on
$\core$ associated with a symmetric nonnegative operator $P$ in $\Gw$.
 We say that a sequence $\{\vgf_k\}\subset\core$  of
nonnegative functions is a {\em null-sequence} of the quadratic form
$q$ in $\Omega$, if there exists an open set $B\Subset\Omega$
such that
$$\lim_{k\to\infty}q(\vgf_k)=0,\qquad \mbox{and } \; \int_B|\vgf_k|^2\dx=1.$$
 We say that a positive function $\phi\in
C^\ga_{\mathrm{loc}}(\Omega)$ is a {\em (Agmon) ground state} of the
functional $q$ in $\Omega$ if $\phi$ is an
$L^2_{\mathrm{loc}}(\Omega)$ limit of a null-sequence of $q$ in $\Gw$.
}
\end{definition}
\begin{definition}\label{def:minimal_gr} {\em
 Let $K\Subset \Omega$, and let $u$ be a positive solution of the equation $Pw=0$ in  $\Omega\sm K$. We say that $u$ is {\em a positive solution of  minimal growth in a neighborhood of infinity in $\Omega$}  if for any $K\Subset K'\Subset \Omega$ with smooth boundary  and any (regular) positive supersolution $v\in C((\Omega\sm K')\cup \,\pd K')$ of the equation $Pw=0$ in $\Omega \sm K'$ satisfying  $u \leq v$ on $\pd K'$, we have $u \leq v$ in $\Omega \sm K'$.
}
  \end{definition}
\begin{theorem}[\cite{PT}]\label{thm_crit}
Suppose that $P$ is nonnegative symmetric operator in $\Gw$, and let $q$ be the corresponding quadratic form.  Then the following assertions are equivalent
\begin{itemize}
\item[(i)] The operator $P$ is critical in $\Gw$.
\item[(ii)] The quadratic form admits a null-sequence and a ground state $\phi$ in $\Gw$.
\item[(iii)] The equation $Pu=0$ admits a unique
positive supersolution $\phi$ in $\Gw$.
\item[(iv)] The equation $Pu=0$ admits a positive solution in $\Gw$ of minimal growth in a neighborhood of infinity in $\Omega$.
\end{itemize}
In particular, any ground state is the unique positive (super)solution of the equation $Pu=0$ in $\Gw$, and it has minimal growth in a neighborhood of $\bar{\infty}$.
\end{theorem}
Let $P$ and $W\gneqq 0$ be as above, the {\em generalized principal eigenvalue} is defined by
\begin{equation}\label{eq_gl_0}
\lambda_0: =\lambda_0(P,W,\Omega):=\sup\Big\{\gl\in\Real\mid P_\gl=P- \gl W\geq 0 \; \mbox{ in }  \Gw \Big\}.
\end{equation}
We also define
$$\lambda_\infty=\lambda_\infty(P,W,\Omega):=\sup\Big\{\gl\in\Real\mid \exists K\subset
\subset\Gw\mbox{ s.t. } P_\gl\geq 0 \; \mbox{ in }  \Gw\setminus K \Big\}.$$
Recall that if the operator $P$ is symmetric in $L^2(\Omega,\dx)$,  and $W>0$, then $\lambda_0$ (resp. $\lambda_\infty$) is the infimum of the $L^2(\Gw, \, W\mathrm{d}x)$-spectrum (resp. $L^2(\Gw, \, W\mathrm{d}x)$-essential spectrum) of the Friedrichs extension of the operator $\tilde{P}:=W^{-1}P$  (see for example \cite{Ag83} and references therein). Note that $\tilde{P}$ is symmetric on $L^2(\Gw, \, W\mathrm{d}x)$, and has the same quadratic form as $P$.
\begin{definition}\label{def_mean_conv}{\em
 Let $\Gw\subsetneqq \R^n$ be a domain. We say that $\Gw$ is {\em  weakly mean convex} if $\gd_{\Gw}$ is weakly superharmonic in $\Gw$.
 }
\end{definition}
Recall that $\gd_{\Gw}\in W^{1,2}_\loc(\Gw)$. Also, any convex domain is of course weakly mean convex, and if $\partial \Gw\in C^2$, then $\Gw$ is weakly mean convex if and only if the mean curvature at any point of $\partial \Gw$ is nonnegative  (see for example \cite{Ps}).

\medskip

Throughout the paper we fix a cone
\begin{equation}\label{def_cone1}
    \Gw:=\{x\in \mathbb{R}^n\mid  r(x)>0,\gw(x)\in \Gs\}\,,
\end{equation}
where $\Gs$ is a Lipschitz domain in the unit sphere $\mathbb{S}^{n-1}\subset \mathbb{R}^n$, $n\geq 2$. For $x\in \Sigma$, we will denote $\mathrm{d}_\Sigma(x)$ the (spherical) distance from $x$ to the boundary of $\Sigma$. Note that $\gd_{\Gw}$ is clearly a homogeneous function of degree $1$, that is,
\begin{equation}\label{eq_homo}
\gd_{\Gw}(x)=|x|\gd_{\Gw}\Big(\frac{x}{|x|}\Big)=r\gd_{\Gw}(\gw).
\end{equation}
Since the distance function to the boundary of any domain is Lipschitz continuous, Euler's homogeneous function theorem implies that
\begin{equation}\label{eq_Eu_homo1}
x\cdot \nabla \gd_{\Gw}(x)= \gd_{\Gw}(x)\hspace{1em}\mbox{a.e. in }\Gw.
\end{equation}
In fact, Euler's theorem characterizes all sufficiently smooth positive homogeneous functions. Hence, (\ref{eq_Eu_homo1}) characterizes the cones in $\Rn$.
For spectral results and Hardy inequalities with homogeneous weights on $\R^n$ see \cite{HOL}.

We note that if $\Gs$ is $C^2$,  then
\begin{equation}\label{eq_homo3}
\delta_\Omega(\gw) =\sin \big(\mathrm{d}_\Sigma(\gw)\big)\; \mbox{ near the boundary of }\Gs.
\end{equation}
Indeed, for $\gw\in \Sigma$, let $z\in \partial\Omega$ such that $|z-\gw|=\gd_\Gw(\gw)$, and let $y\in \partial\Sigma$ realizes ${\rm d}_{\Gs}(\gw)$. Since $\Gs$ is $C^2$, if $\gw$ is close enough to $\partial\Gs$, then $z$ is unique and $\neq 0$, and the points $0,z,y$ are collinear. Moreover, the acute angle between the vectors $\overrightarrow{0y}$ and $\overrightarrow{0\gw}$ is equal to $\mathrm{d}_\Gs(\gw)$. Given that $\overrightarrow{0z}$ is orthogonal to $\overrightarrow{\gw z}$, by elementary trigonometry in the triangle $0,\gw,y$, one gets that $\delta_\Omega(\gw) =\sin \big(\mathrm{d}_\Sigma(\gw)\big)$.

Let $\Gd_S$ be the Laplace-Beltrami operator on the unit sphere $S:=\mathbb{S}^{n-1}$. Then in spherical coordinates, the operator $$P_\gm:=-\Gd-\frac{\gm}{\gd_{\Gw}^{2}} $$ has  the following skew-product form
\begin{equation}\label{eq_pm_spher}
P_\gm u(r,\gw)= -\frac{\partial^2 u}{\partial r^2}-\frac{n-1}{r}\frac{\partial u}{\partial r}+\frac{1}{r^2}\Big(-\Gd_S u-\gm\frac{u}{\gd_\Gw^2(\gw)}\Big) \qquad r>0, \,\gw\in \Gs.
\end{equation}

For any Lipschitz cone the Hardy inequality holds true (as in the case of sufficiently smooth bounded domain \cite{MMP}). We have
\begin{lemma}\label{lem_gl0}
Let $\Gw$ be a Lipschitz cone, and let $\gm_0:=\gl_0(-\Gd,\gd_\Gw^{-2},\Gw)$. Then
\begin{equation}\label{eq_gl0}
0<\gm_0\leq \frac{1}{4} \,.
\end{equation}
In other words, the following Hardy inequality holds true.
\begin{equation}\label{cls_hardy}
 \int_{\Omega}|\nabla \varphi|^2\dx \geq \gm_0\int_{\Omega} \frac{|\varphi|^2}{\gd_\Gw^2}\dx \qquad \forall \varphi\in C_0^\infty(\Omega),
 \end{equation}
where $0<\gm_0\leq 1/4$ is the best constant.

Moreover,  if  $\Gw$ is a weakly mean convex domain, then $\gm_0= {1}/{4}$.
\end{lemma}
\begin{proof}
Using Rademacher's theorem it follows that $\partial \Gw$ admits a tangent hyperplane almost everywhere in $\partial \Gw$. Hence, \cite[Theorem~5]{MMP} implies  that
$$\gm_0=\lambda_0(-\Gd,\gd_\Gw^{-2},\Omega)\leq \lambda_\infty(-\Gd,\gd_\Gw^{-2},\Omega)\leq \frac{1}{4}\,. $$
We claim that $\gm_0>0$. Indeed, denote by $\Gw_R$ the truncated cone
\begin{equation}\label{TC}
\Gw_R:=\{x\in \mathbb{R}^n\mid  0<r<R,\;\; \gw\in \Gs\},
\end{equation}
then
$$0<\gl_{0,R}:=\gl_0(-\Gd,\gd_{\Gw_R}^{-2},\Gw_R),$$
(see for example, \cite{N,MMP}). By comparison,
$$\gm_0\leq\gl_0(-\Gd,\gd_{\Gw}^{-2},\Gw_R), \quad \mbox{and} \quad  0<\gl_{0,R}\leq  \gl_0(-\Gd,\gd_{\Gw}^{-2},\Gw_R).$$
It is well known that if $\{\Gw_k\}$ is an exhaustion of $\Gw$, then
$$\lim_{k\to\infty}  \gl_0(P,W,\Gw_k)=\gl_0(P,W,\Gw).$$
Hence, $$\lim_{R\to\infty}  \gl_0(-\Gd,\gd_{\Gw}^{-2},\Gw_R)=\gm_0 .$$
On the other hand, since $\gd_{\Gw}$ is homogeneous of order $1$, it follows that $\gl_0(-\Gd,\gd_{\Gw}^{-2},\Gw_R)$
is $R$-independent. Therefore,
$$0<\gl_{0,1} \leq \gl_0(-\Gd,\gd_{\Gw}^{-2},\Gw_1)=\gl_0(-\Gd,\gd_{\Gw}^{-2},\Gw_R)=\lim_{R\to\infty}  \gl_0(-\Gd,\gd_{\Gw}^{-2},\Gw_R)=\gm_0 .$$
Consequently,
$$\gm_0=\gl_0(-\Gd,\gd_{\Gw}^{-2},\Gw_R)>0.$$

Assume further that $\Gw$ is a convex cone, or even a weakly mean convex cone. Then it is well known that $\gm_0= {1}/{4}$ (see for example \cite{BFT,MMP}).
\end{proof}
\begin{remark}\label{pr1}{\em
Clearly, $P_\gm$ is subcritical in $\Gw$ for all $\gm<\gm_0$, and by Proposition~\ref{pro_moztkin}, $P_{1/4}$ is subcritical in a weakly mean convex cone. We show in Theorem~\ref{thm_opt_hardy2} that if $\gm_0<1/4$ and $\Gs\in C^2$, then the operator $P_{\gm_0}$ is critical in the cone $\Gw$ (cf. \cite[Theorem~II]{MMP}).
}
\end{remark}
\section{Positive multiplicative solutions}\label{sec_multi}
As above, let $\Gw$ be a Lipschitz cone. By Lemma~\ref{lem_gl0} the generalized principal eigenvalue $\gm_0:=\gl_0(-\Gd,\gd_\Gw^{-2},\Gw)$ satisfies
$0<\gm_0\leq 1/4$.  We have
\begin{theorem}\label{thm_etreme}
Let $\gm\leq\gm_0$. Then the equation $P_\gm u=0$ in $\Gw$ admits positive solutions of the form
\begin{equation}\label{eq_multi1}
    u_\pm(x)=|x|^{\gg_{\pm}}\phi_\gm\Big(\frac{x}{|x|}\Big),
    \end{equation}
where
$\phi_\gm$ is a positive solution of the equation
\begin{equation}\label{eq_pm_spher4}
\Big(-\Gd_S -\frac{\gm}{\gd_\Gw^2(\gw)}\Big)\phi_\gm=\gs(\gm) \phi_\gm \qquad\mbox{in } \Gs,
\end{equation}
\begin{equation}\label{eqgsgm}
    -\frac{(n-2)^2}{4}\leq  \gs(\gm):=\gl_0\Big(-\Gd_S -\frac{\gm}{\gd_\Gw^2}, \mathbf{1},\Gs\Big),
\end{equation}
and
\begin{equation}\label{eq_gb}
\gg_{\pm}:= \frac{2-n \pm\sqrt{(n-2)^2+4\gs(\gm)}}{2}\,.
\end{equation}

Moreover, if $\gs(\gm) > -(n-2)^2/4$, then there are two linearly independent positive solutions of the equation $P_\gm u=0$ in $\Gw$ of the form \eqref{eq_multi1}, and $P_\gm$ is subcritical in $\Gw$.

In particular, for any $\gm\leq \gm_0$ we have $\gs(\gm)>-\infty$.
\end{theorem}

\begin{proof}
We first note that if $u$ is a positive solution of the form  \eqref{eq_multi1}, then clearly $\phi_\gm>0$ and $\phi_\gm$ solves \eqref{eq_pm_spher4}, and $\gg_\pm$ satisfies \eqref{eq_gb}.

Fix a reference point $x_1\in \Gw\cap \mathbb{S}^{n-1}$, and consider an {\em exhaustion} $\{\Gs_k\}_{k=1}^{\infty}\subset\Gs \subset \mathbb{S}^{n-1}$ of $\Gs$ (i.e., $\{\Gs_k\}_{k=1}^{\infty}$ is a sequence of
smooth, relatively compact domains in $\Gs$ such that
$x_1\in \Gs_{k}\Subset \Gs_{k+1}$ for $k\geq 1$, and
$\cup_{k=1}^{\infty}\Gs_{k}=\Gs$).

Fix $\gm\leq\gm_0$. For  $k\geq 1$, and denote the cone
$$ \mathcal{W}_k:=\big\{x\in \mathbb{R}^n\mid  r>0,\;\;\gw\in \Gs_k\big\}\,.$$
Consider the convex set $\mathcal{K}^0_{P_\gm}(\mathcal{W}_k)$ of all positive solutions $u$ of the equation $P_\gm u=0$ in $\mathcal{W}_k$ satisfying the Dirichlet boundary condition $u=0$ on $\partial \mathcal{W}_k\setminus \{0\}$, and the normalization condition $u(x_1)=1$.

Clearly, for $\gm\leq \gm_0$ we have
$$\gm\leq  \lambda_0(-\Gd,\gd_{\Gw}^{-2},\mathcal{W}_k)=\sup\big\{\gl\in\Real\mid \mathcal{K}^0_{P_\gm}(\mathcal{W}_k)\neq \emptyset \big\}.$$
Moreover, $P_\gm$ is subcritical in $\mathcal{W}_k$, and has Fuchsian-type singularities at the origin and at infinity. Hence, in view of \cite[Theorem~7.1]{P94}, it follows that $\mathcal{K}^0_{P_\gm}(\mathcal{W}_k)$, which is a convex compact set in the compact-open topology, has exactly two extreme points.

Next, we characterize the two extreme points of $\mathcal{K}^0_{P_\gm}(\mathcal{W}_k)$ using two different approaches.

\noindent {\bf First method:} We use the results of Section~8 of \cite{LP}. Consider the multiplicative group $\mathcal{G}:=\R^\star$ of all positive real numbers. Then $\mathcal{G}$ acts on $\overline{\mathcal{W}_k}\setminus \{0\}$ (and also on $\overline{\Gw}\setminus \{0\}$) by homotheties $x\mapsto sx$, where $s \in \mathcal{G}$ and $x\in \overline{\mathcal{W}_k}\setminus \{0\}$. This is a compactly generating (cocompact) abelian group action, and $P_\gm$ is an invariant elliptic operator with respect to this action on $\mathcal{W}_k$. In spherical coordinates, a positive $\mathcal{G}$-multiplicative function on $\mathcal{W}_k$ is of the form
\begin{equation}\label{eq_multi}
    f(r,\gw)=r^\gg\phi(\gw),
\end{equation}
where $\gg\in \R$. We note that positive solutions in $\mathcal{K}^0_{P_\gm}(\mathcal{W}_k)$ satisfy a uniform boundary Harnack principle on $\partial\mathcal{W}_k\setminus \{0\}$. Recall that $\mathcal{K}^0_{P_\gm}(\mathcal{W}_k)$ has exactly two extreme points. Hence, by theorems 8.7 and 8.8 of \cite{LP}, $ \lambda_0(-\Gd,\gd_{\Gw}^{-2},\mathcal{W}_k)>\gm$, and  the two extreme points in $\mathcal{K}^0_{P_\gm}(\mathcal{W}_k)$ are positive $\mathcal{G}$-multiplicative solutions of the equation $P_\gm u=0$ in $\mathcal{W}_k$, and therefore, they have the form
\begin{equation}\label{eq_multi2}
    u_{\pm,k}(r,\gw)=r^{\gg_{\pm,k}}\phi_{\pm,k}(\gw).
\end{equation}
In particular, $\phi_{\pm,k}$ vanish on $\Gs_k$.

Using the spherical coordinates representation \eqref{eq_pm_spher} of $P_\gm$, it follows, that $\phi_{\pm,k}$ are positive in $\Gs$, satisfy $\phi_{\pm,k}(x_1)=1$, and solve the eigenvalue Dirichlet problem
\begin{equation}\label{eq_pm_spher1}
\Big(-\Gd_S -\frac{\gm}{\gd_\Gw^2(\gw)}\Big)\phi_{\pm,k}=\big(\gg_{\pm,k}^2+\gg_{\pm,k}(n-2)\big)\phi_{\pm,k} \;\; \mbox{in } \Gs_k, \;\; \phi_\pm=0 \;\;\mbox{on } \partial \Gs_k.
\end{equation}
On the other hand, since the operator $-\Gd_S -\gm\gd_\Gw^{-2}$ has up to the boundary regular coefficients in $\Gs_k$, it admits a unique (Dirichlet) eigenvalue $\gs_k$ with a positive eigenfunction $\phi_k$ satisfying $\phi_k(x_1)=1$. Moreover, $\gs_k$ is simple. In other words, $\gs_k$ and $\phi_k$ are respectively the {\em principal} eigenvalue and eigenfunction of
$-\Gd_S -\gm\gd_\Gw^{-2}$ in $\Gs_k$.

Hence, $\phi_{\pm,k}$ are equal to $\phi_k$, and  $$\gs_k: =\gs_k(\gm)=\big(\gg_{\pm,k}^2+\gg_{\pm,k}(n-2)\big).$$

By the strict monotonicity with respect to bounded domains of the principal eigenvalue of second-order elliptic operators with up to the boundary regular coefficients, it follows that $\gs_k(\gm)> \gs_{k+1}(\gm)$.

\medskip

On the other hand, since
\begin{equation}\label{eq_multi21}
    u_{\pm,k}(r,\gw))=r^{\gg_{\pm,k}}\phi_k(\gw)>0,
\end{equation}
it follows that $\gg_{-,k}\neq \gg_{+,k}$, and $\gg_{\pm,k}$ are given by
$$\gg_{\pm,k}:= \frac{2-n \pm\sqrt{(n-2)^2+4\gs_k}}{2}\,.$$
In particular,
$$ \gg_{-,k}<\gg_{-,k+1}<\frac{2-n}{2}<\gg_{+,k+1}<\gg_{+,k} \quad \mbox{and } \gs_k > -\frac{(n-2)^2}{4}\,.$$

\medskip

\noindent {\bf Second method:} We only indicate briefly the second approach. We use the results of \cite{M}. By \eqref{eq_pm_spher}, the subcritical elliptic operator $P_\gm$ has a skew-product form in $\mathcal{W}_k=\R_+\times \Gs_k$ and satisfies the conditions of Theorem~1.1 of \cite{M}. Therefore, the equation $P_\gm u=0$ admits two Martin functions of the form \eqref{eq_multi2}.

\medskip

Now, let $k \to \infty$. Then $\gs_k \searrow \gs \geq -(n-2)^2/4$, and up to a subsequence $\phi_k\to \phi_\gm$ locally uniformly in $\Gs$. Clearly, $\gs$ does not depend on the exhaustion of $\Gs$. Recall also that for any nonnegative second-order elliptic operator $L$ in a domain $D$ and any exhaustion $\{D_k\}$ of $D$ we have
$$\gl_0(L,W,D)=\lim_{k\to\infty}\gl_0(L,W,D_k).$$ Hence,  $\gs=\gs(\gm)=\gl_0\big(-\Gd_S -\gm\gd_\Gw^{-2}, \mathbf{1},\Gs\big)$.

Consequently,
$\gg_{\pm,k} \to \gg_{\pm}$, where $\gg_{-}\leq -(n-2)/2 \leq \gg_{+}$.
Hence, we have that
$$\lim_{k\to\infty}u_{\pm,k}(r,\gw))=\lim_{k\to\infty}r^{\gg_{\pm,k}}\phi_k(\gw)=r^{\gg_{\pm}}\phi_\gm(\gw).$$

If $\gg_{-}<-(n-2)/2<\gg_{+}$ (or equivalently, $\gs(\gm) > -(n-2)^2/4$ ), then we obtain two linearly independent $\mathcal{G}$-multiplicative positive solutions of the equation $P_\gm u=0$ in $\Gw$. In particular, $P_\gm$ is subcritical in $\Gw$.
\end{proof}

\begin{remark}\label{Hardy_RM1}{\em
Note that for $n= 2$, $\Gs=\mathbb{S}^{1}$,  and $\gm=\gm_0=0$, we obtain $\gs(0)=0$, $\gg_{\pm}=0$, and $P_0=-\Gd$ is critical in the cone $\R^2\setminus\{0\}$.
 }
\end{remark}

\begin{remark}\label{Hardy_RM}{\em
Let $\Gs$ be a bounded domain in a smooth Riemannian manifold $M$, and let $\mathrm{d}_\Gs$ be the Riemannian distance function to the boundary $\partial \Gs$. If $\Gs$ is smooth enough, then the Hardy inequality with respect to the weight $(\mathrm{d}_\Gs)^{-2}$  holds in $\Gs$ with a positive constant $C_H$ \cite{R}. A sufficient condition for the validity of a such Hardy inequality is that $\Gs$ is {\em boundary distance regular}, and this condition holds true if $\Gs$ satisfies either the {\em uniform interior cone condition} or the {\em  uniform exterior ball condition} (see the definitions in \cite{R}). For other sufficient conditions for the validity of the Hardy inequality on Riemannian manifolds see for example \cite{MV}.

Hence, if the cone $\Gw\subsetneqq \R^n\setminus\{0\}$ is smooth enough, then $\Gs\subset \mathbb{S}^{n-1}$ is boundary distance regular.
So, for such $\Gs\subset \mathbb{S}^{n-1}$, there exists $C>0$ such that $-\Gd_S -C\mathrm{d}_\Gs^{-2}\geq 0$ in $\Gs$. Note that $\mathrm{d}_\Gs(\gw)\asymp \gd_\Gw(\gw)|_\Gs$ in $\Gs$, therefore,
$-\Gd_S -C_1\gd_\Gw^{-2}\geq 0$ in $\Gs$ for some $C_1>0$.
 }
\end{remark}
In the sequel we shall need the following lemma concerning the criticality of the operator $\mathcal{L}_\gm:=-\Gd_S -\gm\gd_\Gw^{-2}-\gs(\gm)$ in $\Gs$.
\begin{Lem}\label{lem_pos_crit}
Consider the operator $\mathcal{L}_\gm=-\Gd_S -\gm\gd_\Gw^{-2}-\gs(\gm)$ on $\Sigma$. Then
\begin{enumerate}
\item We have
\begin{equation}\label{mu_0}
\mu_0=\lambda_0\Big(-\Delta_S+\frac{(n-2)^2}{4},\delta_\Omega^{-2},\Sigma\Big).
\end{equation}
\item Assume that $\Gs\in C^2$, and $\mu_0<1/4$. Then $\sigma(\mu_0)=-(n-2)^2/4$, and $\mathcal{L}_{\gm_0}$ is critical in $\Sigma$ with ground state $\phi_{\mu_0}\in L^2(\Sigma,\delta_\Omega^{-2}\mathrm{d}S)$.
\item Assume that $\Gs\in C^2$, and $\mu_0=1/4$. Then $\mathcal{L}_{1/4}$ is critical in $\Gs$ with ground state $\phi_{1/4}\in L^2$ $(\Gs,\gd_\Gw^{-2}\log(\gd_\Gw)^{-(1+\ge)}\mathrm{d}S)$, where $\ge$ is any positive number.
\item Assume that $\mu<\mu_0$, then $\mathcal{L}_{\gm}$ is positive critical in $\Gs$. That is, $\mathcal{L}_{\gm}$ admits a ground state $\phi_\gm$ in $\Gs$, and $\gf_\gm\in L^2(\Gs)$.
\end{enumerate}
In particular, in all the above cases, $\phi_\gm$ is (up to a multiplicative constant) the unique positive (super)solution of the equation $\mathcal{L}_{\gm}u=0$ in $\Gs$, and $\gf_\gm\in L^2(\Gs)$.
\end{Lem}
\begin{proof}
1. To prove \eqref{mu_0} we note that Theorem~\ref{thm_etreme} implies that for $\mu\leq \mu_0$ there exists $\phi_\mu$ positive solution  of
$$\mathcal{L}_{\gm}u= \Big(-\Delta_S-\frac{\mu}{\delta_\Omega^2}-\sigma(\gm)\Big)u=0 \qquad \mbox{in } \Sigma, $$
and since for any $\mu\leq \mu_0$, we have $\sigma(\mu)\geq -(n-2)^2/4$, it follows that $\phi_\mu$ is a positive supersolution of the equation
$$\mathcal{L}_{\gm}u= \Big(-\Delta_S-\frac{\mu}{\delta_\Omega^2}+\frac{(n-2)^2}{4}\Big)u=0 \qquad \mbox{in } \Sigma. $$
 Thus, by the AAP Theorem (Theorem~\ref{thm_AAP}) we get,
$$\mu_0\leq \lambda_0\Big(-\Delta_S+\frac{(n-2)^2}{4}\!,\delta_\Omega^{-2},\Sigma\Big).$$
Let us now take $\mu>\mu_0$, and assume by contradiction that  $-\Delta_S+(n-2)^2/4-\mu\delta_\Omega^{-2}\geq 0$ in $\Gs$. Then by definition, there is a positive solution $\phi_\mu$  of the equation
$$\Big(-\Delta_S-\frac{\mu}{\delta_\Omega^2}+\frac{(n-2)^2}{4}\Big)u=0 \qquad \mbox{in } \Sigma.$$
If one defines
$$\psi(x)=|x|^{(2-n)/2}\phi_\mu\Big(\frac{x}{|x|}\Big),$$
then it is immediate to check that $\psi$ is a positive solution in $\Omega$ of
$$\Big(-\Delta-\frac{\mu}{\delta^2_\Omega}\Big)u=0\qquad \mbox{in } \Gw .$$
This implies that
$$\lambda_0\left(-\Delta,\delta_\Omega^{-2},\Omega\right)\geq \mu>\mu_0,$$
a contradiction. Thus, the operator $-\Delta_S+(n-2)^2/4-\gm\delta_\Omega^{-2}$ cannot be nonnegative in $\Gs$  for $\mu>\mu_0$, and this implies that
$$\mu_0\geq \lambda_0\Big(-\Delta_S+\frac{(n-2)^2}{4},\delta_\Omega^{-2},\Sigma\Big).$$
Hence, \eqref{mu_0} is proved.

2. Since
$$\mathrm{d}_\Sigma(x)\sim \delta_\Omega(x)\qquad  \mbox{as } x\in \Sigma,\,\mathrm{d}_\Sigma(x)\to0,$$
and in light of the proof of \cite[Theorem 5]{MMP}, our assumption that $\Sigma$ is $C^2$ implies that
$$\lambda_\infty\left(-\Delta_S,\delta_\Omega^{-2},\Sigma\right)=\frac{1}{4},$$
which in turn implies that
$$\lambda_\infty\Big(-\Delta_S+\frac{(n-2)^2}{4},\delta_\Omega^{-2},\Sigma\Big)=\frac{1}{4}.$$
On the other hand, by part 1 we have
$$\lambda_0\Big(-\Delta_S+\frac{(n-2)^2}{4},\delta_\Omega^{-2},\Sigma\Big)=\gm_0.$$
Hence, our assumption that $\gm_0<1/4$, implies that there is a spectral gap between the bottom of the $L^2(\Sigma,\gd_\Gw^{-2}\mathrm{d}S)$-spectrum and the bottom of the essential spectrum of the operator $-\Delta_S+(n-2)^2/4$ in $\Sigma$. Consequently, the operator $-\Delta_S+(n-2)^2/4-\gm_0\delta_\Omega^{-2}$
is critical in $\Sigma$, with ground state $\phi_{\mu_0}\in L^2(\Sigma,\delta_\Omega^{-2}\mathrm{d}S)$. Clearly, the criticality of $-\Delta_S+(n-2)^2/4-\mu_0\delta_\Omega^{-2}$  in $\Sigma$ implies that
$$\sigma(\mu_0)=-\frac{(n-2)^2}{4}\,,$$
and the second part of the lemma is proved.

\medskip

\noindent Before proving part 3, we prove the fourth part of the lemma.

\medskip

4. The assumption $\mu<\mu_0$ clearly implies that $\gl_\infty\left(-\Gd_S -\gm\gd_\Gw^{-2}, \mathbf{1},\Gs\right)=\infty$. Hence,
 $$-\frac{(n-2)^2}{4}\leq \gs(\gm)=\gl_0\Big(-\Gd_S -\frac{\gm}{\gd_\Gw^2}, \mathbf{1},\Gs\Big)<\gl_\infty\left(-\Gd_S -\gm\gd_\Gw^{-2}, \mathbf{1},\Gs\right)=\infty.$$
Since $\gl_0$ (respect. $\gl_\infty$) is the bottom of the (respect. essential) $L^2$-spectrum of the operator $-\Gd_S -\gm\gd_\Gw^{-2}$ in $\Gs$, it follows that the operator $\mathcal{L}_\gm$ is critical in $\Gs$, and $\gs(\gm)$ is the principal eigenvalue of the operator
$-\Gd_S -\gm\gd_\Gw^{-2}$ with principal eigenfunction $\gf_\gm\in L^2(\Gs)$. Hence, the operator $\mathcal{L}_\gm$ is positive critical in $\Gs$.

\medskip

3. The proof uses a modification of Agmon's trick (\cite[Theorem 2.7]{Ag85}, see also \cite[Lemma 7]{MMP}). In order to prove that $\gl_\infty(-\Gd_S-1/(4\gd_\Gw^2),\mathbf{1},\Gs)=\infty$, we will show that for suitable positive constants $c,\vge,$ the function $\gd_\Gw^{1/2}-\gd_\Gw/2$ is a positive supersolution of the equation
\begin{equation}\label{eq_trick}
    \left(-\Gd_S-\frac{1}{4\gd_\Gw^2}-\frac{c}{\gd_\Gw^{\vge}}\right)u=0
\end{equation}
 in a sufficiently small neighborhood of the boundary of $\Gs$.

We start by denoting a tubular neighborhood of $\partial\Gs$ having width $\gb>0$, by
 \be\nonumber
  \Gs_\gb:=\{\gw\in \Gs\;|\;{\rm d}_\Gs(\gw)<\gb\}.
 \ee
Recall that since $\Sigma$ is $C^2$, there exists $\gb_*>0$ such that $\mathrm{d}_\Sigma\in C^2$ in $\Gs_{\gb_*}$. In particular, $-\Gd_S{\rm d}_\Gs$ is bounded on $\Gs_{\gb_*}$. Also $|\nabla_S{\rm d}_\Gs|=1$ and $\gd_\Gw=\sin({\rm d}_\Gs)$ (by \eqref{eq_homo3}), both on $\Gs_{\gb_*}$. We may thus compute
 \be\nonumber
  -\Gd_S\gd_\Gw
   =
    \cos({\rm d}_\Gs)\Gd_S{\rm d}_\Gs-\sin({\rm d}_\Gs)
     \;\;\;\mbox{ on }\Gs_{\gb_*},
 \ee
which implies that $\Gd_S\gd_\Gw$ is also bounded on $\Gs_{\gb_*}$. In particular, we have
 \be\label{estim:lapd}
  -\Gd_S\gd_\Gw(\gw)
   \geq
    -h
     \;\;\;\mbox{ for all }\gw\in\Gs_{\gb_*},
 \ee
for some $h>0.$ Now let $c,\vge>0$ and compute on $\Gs_{\gb_*}$
\begin{align*}\nonumber
& \Big(-\Gd_S-\frac{1}{4\gd_\Gw^2}-\frac{c}{\gd_\Gw^\vge}\Big)\Big(\gd_\Gw^{1/2}-\frac{\gd_\Gw}{2}\Big)
 \\[2mm] & =
  -\frac{1}{4\gd_\Gw^{3/2}}(1-|\nabla_S\gd_\Gw|^2)-\frac{1}{2\gd_\Gw^{1/2}}(1-\gd_\Gw^{1/2})\Gd_S\gd_\Gw+\frac{1}{8\gd_\Gw}-c\gd_\Gw^{1/2-\vge}+\frac{c}{2}\gd_\Gw^{1-\vge}
   \\[2mm] & \geq
    -\frac{\gb^2_*}{4\gd_\Gw^{1/2}}-\frac{h}{2\gd_\Gw^{1/2}}(1-\gd_\Gw^{1/2})+\frac{1}{8\gd_\Gw}-c\gd_\Gw^{1/2-\vge}+\frac{c}{2}\gd_\Gw^{1-\vge},
\end{align*}
where we have used the fact that $1-|\nabla_S\gd_\Gw|^2=\sin^2({\rm d}_\Gs)\leq\gb^2_*$ on $\Gs_{\gb_*}$ and also (\ref{estim:lapd}). Clearly, by fixing $\vge$ in $(0,3/2)$ we obtain that this estimate blows up as $\gw\in\Gs_{\gb_*}$ approaches the boundary of $\Gs$. Thus, for a smaller $\gb_*>0$ if necessary, we proved that $\gd_\Gw^{1/2}-\gd_\Gw/2$ is a positive supersolution of \eqref{eq_trick} in $\Gs_{\gb_*}$. The APP theorem (Theorem \ref{thm_AAP}) implies
 \be\label{log-h}
  \int_{\Gs_{\gb_*}}\Big(|\nabla u|^2-\frac{1}{4\gd_\Gw^2}\Big)\vgf^2{\rm d}S
   \geq
    c\int_{\Gs_{\gb_*}}\frac{\vgf^2}{\gd^\vge_\Gw}{\rm d}S\;\;\;\;\;\;\forall \vgf\in C_0^\infty(\Gs_{\gb_*}),
 \ee
which together with $\lim_{{\rm d}_\Gs(\gw)\to 0}\gd_\Gw^{-\vge}(\gw)=\infty$ imply that
$$\lambda_\infty\Big(\Gd_S-\frac{1}{4\gd_\Gw^2},\mathbf{1},\Gs\Big)=\infty.$$
As in the proof of part 2, one concludes that $\mathcal{L}=\-\Delta_S-1/(4\delta_\Omega^2)-\sigma(\mu)$ is critical, with ground state $\phi_{1/4}\in L^2(\Sigma)$.

It remains to show that in fact, $\phi_{1/4}\in L^2(\Sigma,\delta_\Omega^{-2}\log^{-(1+\ge)}(\delta_\Omega)\mathrm{d}S)$. In fact, the arguments used in the proof of \cite[Lemma~9]{MMP} show that, as $\gw\in \Gs$ and $\delta_\Omega(\gw)\to0$,
$$\phi_{1/4}(\gw)\asymp \delta_\Omega^{1/2}(\gw).$$
This implies that $\phi_{1/4}\in L^2(\Sigma,\delta_\Omega^{-2}\log^{-(1+\ge)}(\delta_\Omega)\mathrm{d}S)$ for any $\ge>0$.
\end{proof}
\begin{proposition}\label{pro_7} Let $\gs(\gm)=\lambda_0(-\Delta_S-\gm\gd_\Gw^{-2},\mathbf{1},\Sigma)$. Then
\begin{enumerate}
\item $\gs(\gm)\geq  -(n-2)^2/4$ for any $\gm\leq \gm_0$, and if $\Gs\in C^2$ and $\mu_0<1/4$, then  $\gs(\gm_0) = -(n-2)^2/4$.

\item $\gs(\gm)=-\infty$ for any $\gm > 1/4$.

\item If $\Gs\in C^2$, then $\gs(\mu)>-\infty$ for all $\gm\leq 1/4$.
\end{enumerate}
\end{proposition}
\begin{proof}
1. Recall that by Lemma~\ref{lem_gl0} we have that $0<\gm_0\leq 1/4$, and by Theorem~\ref{thm_etreme} $\gs(\gm)\geq -(n-2)^2/4$ for all  $\gm \leq \gm_0$. Moreover, by Lemma~\ref{lem_pos_crit}, if $\Gs\in C^2$
 and $\mu_0<1/4$, then $\gs(\gm_0) = -(n-2)^2/4$.  In particular,
for such $\gm$ we have that  $\gs(\gm)$  is finite.

\medskip

2. Let $\gm > 1/4$, and suppose that $\sigma(\mu)$ is finite. Then one can find a positive function $\gf$ satisfying
$$(-\Delta_S-\mu\delta_\Gw^{-2} -\sigma(\mu))\gf=0 \qquad \mbox{in } \Sigma.$$ Take $\varepsilon>0$ such that $\gm-\vge>1/4$. Clearly,
$$\lim_{\gw\to \partial \Sigma}\delta_\Gw^{-2}(\gw)=\infty, \quad \mbox{and }\;  \lim_{\gw\to \partial \Gs}\frac{\delta_\Omega(\gw)}{\mathrm{d}_\Sigma(x)}=1,$$
where $\mathrm{d}_\Gs$ is the Riemannian distance to the boundary of $\Gs$.
Hence, $\gf$ is a positive supersolution of the equation
$$(-\Delta_S-(\mu-\varepsilon)\mathrm{d}_\Gs^{-2})u=0$$
in a neighborhood of infinity in $\Sigma$.

On the other hand, as in \cite{MMP}, if $\Gs$ is a Lipschitz  domain, then
$\lambda_\infty(-\Delta_S,\mathrm{d}_\Sigma^{-2},\Sigma) \leq 1/4$. Consequently, for  such $\vge$, one gets a contradiction to $\lambda_\infty(-\Delta_S,\mathrm{d}_\Sigma^{-2},\Sigma)\leq 1/4$.

\medskip

3.  Suppose first that $\gm<  1/4$. Recall that since $\Gs\in C^2$ we have
$$\lambda_\infty(-\Delta_S,\gd_\Gw^{-2},\Sigma) =\lambda_\infty(-\Delta_S,\mathrm{d}_\Sigma^{-2},\Sigma)=1/4.$$

Take $\vge>0$ such that $\gm+\vge<  1/4$. Let $\gf$ be a positive solution of the equation
$$\big(-\Delta_S-(\mu+\vge)\delta_\Gw^{-2}\big)u=0$$ in a neighborhood of infinity in $\Sigma$, and let $\tilde{\gf}$ be a nice positive function in $\Gs$ such that $\tilde{\gf}=\gf$ in a neighborhood of $\partial \Gs$. Then for $\gs$ large enough, $\tilde{\gf}$ is a positive supersolution of the equation
$(-\Delta_S-\mu\delta_\Gw^{-2}+\gs)u=0$ in $\Sigma$. Hence $\gs(\mu)>-\infty$ for all $\gm<1/4$.

Suppose now that $\gm=  1/4$. By \eqref{eq_trick}, $\psi:=\gd_\Gw^{1/2}-\gd_\Gw/2$ is a positive supersolution of
$$\Big(-\Delta_S-\frac{1}{4\delta_\Omega^2}-\frac{c}{\gd_\Gw^{\ge}}\Big)u=0$$
outside a compact set $K_\varepsilon\Subset \Sigma$. Let $\tilde{\psi}$ be a nice positive function in $\Gs$ such that $\tilde{\psi}=\psi$ in a neighborhood of $\partial \Gs$. Hence, for $\gs$ large enough, $\tilde{\psi}$ is a positive supersolution of the equation
$(-\Delta_S-1/4\delta_\Gw^{2}+\gs)u=0$ in $\Sigma$. Hence $\gs(1/4)>-\infty$.
\end{proof}
\begin{remark}\label{rem_bilip}
{\em In Lemma~\ref{lem_pos_crit} and Proposition~\ref {pro_7}, it is assumed that $\Gs\in C^2$. The extension of the proposition to the class of Lipschitz domains remains open. We recall that by the recent result of G.~Barbatis and  P.~D.~Lamberti \cite[Proposition~1]{BL}, the Hardy constant of a bounded domain is Lipschitz continuous as a function of bi-Lipschitz maps that approximate the domain. It seems that finding for a given Lipschitz domain a uniform bi-Lipschitz smooth approximation is a nontrivial problem: we note that in  \cite[Theorem~1]{DP}, the authors prove the existence of approximation of Lipschitz homeomorphisms by smooth ones in the $W^{1,p}$ topology for $p<\infty$. However, to apply the results in \cite{BL}, we should need $W^{1,\infty}$-approximations.
 }
 \end{remark}
We conclude the present section with the following general result that provides us with a sufficient condition for the criticality of a Schr\"odinger operator on a precompact domain. For a general sufficient condition see \cite{P07}.
\begin{Lem}\label{null-seq}
Let $P=-\Delta+V$ be a nonnegative Schr\"{o}dinger operator on a compact Riemannian manifold with boundary $M$, endowed with its Riemannian measure $\dx$. Denote by $\delta=\gd_M$ the distance function to the boundary of $M$. Assume that $M\in C^2$, $V$ is smooth in the interior of $M$, and that the equation $Pu=0$ in $M$ admits a positive solution $\phi\in L^2(M,\delta^{-2}\log^{-2}(\delta)\dx)$. Then, $P$ is critical in $M$ with ground state $\phi$, and furthermore, there exists a null-sequence $\{\phi_k\}_{k=0}^\infty$ for $P$, which converges locally uniformly and in $L^2$ to $\varphi$.

\end{Lem}

\begin{proof}
If $q$ denotes the quadratic form of $P$, then using the ground state transform (see for example \cite{DFP}) we have for every $\vgf\in C_0^\infty(M)$,
$$q(\phi \vgf)=\int_M \phi^2|\nabla \vgf|^2 \dx.$$
This formula extends easily to every Lipschitz continuous function $\vgf$ which is compactly supported in $M$. For $k\geq2$, let us define $v_k:\R_+\to [0,1]$ by
$$v_k(t)=\left\{
           \begin{array}{ll}
             0 & \;0\leq t\leq 1/k^2, \\[2mm]
             1+\dfrac{\log(kt)}{\log k} & \;1/k^2<t<1/k\,, \\[3mm]
             1 & \;t\geq1/k.
           \end{array}
         \right.
$$
Note that $0\leq v_k(\delta)\leq 1$, and $\{v_k(\delta)\}_{k\geq 2}$ converges pointwise to the constant function $\mathbf{1}$ in $M$. Define
$$\phi_k:=v_k(\delta)\phi,$$
then, using that $\phi\in L^2_\loc$, one sees that $\{\phi_k\}_{k=0}^\infty$ converges locally uniformly and hence in $L^2_{\loc}$ to $\phi$. We now prove that $\{\phi_k\}_{k=2}^\infty$ is a null-sequence for $P$, which implies that $P$ is critical with ground state $\phi$. If $K\Subset M$ is a fixed precompact open set, then clearly, there is a positive constant $C$ such that, for $k$ big enough,
$$\int_K\phi_k^2\dx\asymp 1.$$
Thus, in order to prove that $\{\phi_k\}_{k=2}^\infty$ is a null-sequence for $P$, it is enough to prove that
\begin{equation}\label{nul-seq}
\lim_{k\to\infty}\int_M \phi^2|\nabla v_k(\delta)|^2\mathrm{d}x=0.
\end{equation}
Since $|\nabla \delta(x)|\leq1$ a.e. in $M,$ it is enough to show that
$$\lim_{k\to\infty}\int_M \phi^2|v_k'(\delta)|^2\dx=0.$$
We compute
\begin{equation*}
\int_M \phi^2|v_k'(\delta)|^2\dx
=\int_{\{1/k^2< \delta<1/k\}} \Big(\frac{\phi}{\delta\log(k)}\Big)^2\,\mathrm{d}x \leq 4\int_{\{\delta<1/k\}}\Big(\frac{\phi}{\delta\log(\delta)}\Big)^2\,\mathrm{d}x.
\end{equation*}
By our hypothesis, the function $\phi^2\delta^{-2}\log^{-2}(\delta)$ is integrable on $\{\delta<1/2\}$, hence,
$$\lim_{k\to\infty}\int_{\{\delta<1/k\}}\Big(\frac{\phi}{\delta\log(\delta)}\Big)^2\dx=0,$$
which shows \eqref{nul-seq}. Thus, $\{\phi_k\}_{k\geq 2}$ is a null-sequence for $P$.
\end{proof}
\section{The structure of $\mathcal{K}^0_{P_\gm}(\Gw)$}\label{nsec_multi}
As above, let $\Gw$ be a Lipschitz cone. By Lemma~\ref{lem_gl0} the generalized principal eigenvalue $\gm_0:=\gl_0(-\Gd,\gd_\Gw^{-2},\Gw)$ satisfies
$0<\gm_0\leq 1/4$.

For $\gm\leq \gm_0$, denote by  $\mathcal{K}^0_{P_\gm}(\Gw)$ the convex set of all positive solutions $u$ of the equation $P_\gm u=0$ in $\Gw$ satisfying the normalization condition $u(x_1)=1$, and the Dirichlet boundary condition $u=0$ on $\partial \Gw\!\setminus \!\{0\}$ in the sense of the Martin boundary. That is, any $u\in \mathcal{K}^0_{P_\gm}(\Gw)$ has minimal growth on $\partial \Gw\!\setminus \!\{0\}$. For the definition of minimal growth on a portion $\Gg$ of $\partial \Gw$, see \cite{P94}.

If  $\gm_0<1/4$ and $\Gs$ is $C^2$, then in Theorem~\ref{thm_opt_hardy2} (to be proved in the sequel) we show that the operator $P_{\gm_0}$ is critical in $\Gw$, and therefore the  equation  $P_{\gm_0}u=0$ in $\Gw$ admits (up to a multiplicative constant) a unique positive supersolution. Moreover, by Theorem~\ref{thm_etreme}, the unique positive solution is a multiplicative solution of the form \eqref{eq_multi1}.

The following theorem characterizes the structure of $u\in \mathcal{K}^0_{P_\gm}(\Gw)$ for any $\gm<\gm_0$.
\begin{theorem}\label{nthm_etreme}
Let $\gm<\gm_0 \leq 1/4$. Then $\mathcal{K}^0_{P_\gm}(\Gw)$ is the convex hull of two linearly independent positive solutions of the equation $P_\gm u=0$ in $\Gw$ of the form
\begin{equation}\label{neq_multi1}
    u_\pm(x)=|x|^{\gg_{\pm}}\phi_\gm\Big(\frac{x}{|x|}\Big),
    \end{equation}
where
$\phi_\gm$ is the unique positive solution of the equation
\begin{equation}\label{neq_pm_spher4}
\Big(-\Gd_S -\frac{\gm}{\gd_\Gw^2(\gw)}\Big)\phi_\gm=\gs(\gm) \phi_\gm \qquad\mbox{in } \Gs,
\end{equation}
\begin{equation}\label{neqgsgm}
    -\frac{(n-2)^2}{4}< \gs(\gm):=\gl_0\Big(-\Gd_S -\frac{\gm}{\gd_\Gw^2}, \mathbf{1},\Gs\Big), \mbox{ and}
\end{equation}
\begin{equation}\label{neq_gb}
\gg_{\pm}:= \frac{2-n \pm\sqrt{(2-n)^2+4\gs(\gm)}}{2}\,.
\end{equation}
\end{theorem}

\begin{proof}
The assumption $\gm<\gm_0$ implies that the operator $P_\gm$ is subcritical in $\Gw$.
In particular, $\mu<1/4$, and therefore, there exists $\vge>0$ such that the operator $P_{\mu+\vge}$ is subcritical in a small neighborhood of a given portion of $\partial \Gw\setminus\{0\}$.  Since the operator $P_\gm$ and the cone $\Gw$ are invariant under scaling, it follows from the local Harnack inequality, and from the boundary Harnack principle of A.~Ancona for the operator $P_\mu$ in $\Gw$ \cite{An87} (see also \cite{An12a}) that the following uniform boundary Harnack principle holds true in the annulus $A_R\subset\Gw$. There exists $C>0$ (independent of $R$) such that
\begin{equation}\label{UHI}
C^{-1}\frac{v(x)}{v(y)}\leq C^{-1}\frac{u(x)}{u(y)}\leq C\frac{v(x)}{v(y)} \qquad \forall x,y \in  A_R,
\end{equation}
for any $u,v\in \mathcal{K}^0_{P_\gm}(\Gw)$ and $R>0$.

Hence, we can use directly the arguments in \cite{P94} to obtain that in the subcritical case the convex set $\mathcal{K}^0_{P_\gm}(\Gw)$ has exactly two extreme points. Moreover, we can use directly the method of \cite[Section~8]{LP}, to obtain that $u$ is an extreme point of $\mathcal{K}^0_{P_\gm}(\Gw)$ if and only if it is a positive multiplicative solution in $\mathcal{K}^0_{P_\gm}(\Gw)$. Thus, the two extreme points of $\mathcal{K}^0_{P_\gm}(\Gw)$ are of the form
$$u_\pm(x)=|x|^{\gg_{\pm}}\phi_\pm\Big(\frac{x}{|x|}\Big),$$ where $\phi_\pm>0$ in $\Gs$, and solves the equation
\begin{equation}\label{neq_pm_spher41}
\Big(-\Gd_S -\frac{\gm}{\gd_\Gw^2(\gw)}\Big)\phi_\pm=\gs_\pm \phi_\pm \qquad\mbox{in } \Gs,
\end{equation}
\begin{equation}\label{neqgsgm1}
    -\frac{(n-2)^2}{4}\leq \gs_\pm \leq\gs(\gm):=\gl_0\Big(-\Gd_S -\frac{\gm}{\gd_\Gw^2}, \mathbf{1},\Gs\Big), \mbox{ and}
\end{equation}
\begin{equation}\label{neq_gb1}
\gg_{\pm}:= \frac{2-n \pm\sqrt{(n-2)^2+4\gs_\pm}}{2}\,.
\end{equation}
If $\gg_+=\gg_-$, then \eqref{UHI} implies that $u_+\asymp u_-$. Since $u_\pm(x)$ are two extreme points, and $\mathcal{K}^0_{P_\gm}(\Gw)$ has exactly two extreme points, it follows that $\gg_+\neq \gg_-$. Therefore,  $\gs_\pm =\gs$, where
$-(n-2)^2/4< \gs \leq\gs(\gm)$ and $\gg_\pm$ satisfy
\begin{equation}\label{neq_gb17}
\gg_{\pm}:= \frac{2-n \pm\sqrt{(n-2)^2+4\gs}}{2}\,.
\end{equation}
Moreover, since $\phi_\pm$ solve the same equation in $\Gs$, and $\mathcal{K}^0_{P_\gm}(\Gw)$ has exactly two extreme points, it follows that $\phi_\pm=\phi$.

Note that by Lemma~\ref{lem_pos_crit}, $\phi$ is a positive solution of minimal growth near $\partial \Gs$ if and only if $\gs =\gs(\gm)$. On the other hand, $u_\pm$ have minimal growth near $\partial \Gw\setminus \{0\}$. Therefore, $\phi=\phi_\gm$ and $\gs =\gs(\gm)$, where $\phi_\gm$ is a ground state satisfying \eqref{neq_pm_spher4}, and $\gs(\gm)$ and $\gg_\pm$ satisfy \eqref{neqgsgm} and \eqref{neq_gb}, respectively.
\end{proof}
\section{The main result}\label{sec_main}
The present section is devoted to our main result concerning the existence of an optimal Hardy weight for the operator $P_\gm$ which is defined in a cone $\Gw$. In Theorem~\ref{thm_opt_hardy1} we prove the case where $\gm<\gm_0$ and $\Gw$ is a Lipschitz cone, while in Theorem~\ref{thm_opt_hardy2} we prove the case $\gm=\gm_0$ under the assumption that $\Gs\in C^2$.

Let us recall that by Theorem~\ref{thm_etreme}, if $\mu\leq \mu_0$, then
$$\sigma(\mu):=\lambda_0\Big(-\Delta-\frac{\mu}{\delta_\Omega^2},\mathbf{1},\Sigma\Big)\geq -\frac{(n-2)^2}{4},$$
and there exists a positive solution $\phi_\mu$ of the equation
$$\Big(-\Delta_S-\frac{\mu}{\delta_\Omega^2}-\sigma(\mu)\Big)u=0\qquad \mbox{in } \Sigma.$$
 Furthermore, by Lemma~\ref{lem_pos_crit}, the operator
$$\mathcal{L}:=\mathcal{L}_\gm=-\Delta_S-\frac{\mu}{\delta_\Omega^2}-\sigma(\mu)$$
is critical (for any $\gm<\gm_0$, and also for $\gm=\gm_0$ if in addition $\Gs\in C^2$), and $\phi_\mu$ is the ground state of $\mathcal{L}$.

We first prove.
\begin{Pro}\label{Hardy-w}
Let $\Omega$ be a Lipschitz cone. Let $\mu\leq \mu_0$, and let
\begin{equation}\label{eqgsgm1}
     \gl(\gm):=\frac{(2-n)^2+4\gs(\gm)}{4}\,.
\end{equation}
Then $\lambda(\mu)\geq0$, and the following Hardy inequality holds true in $\Omega$:
\begin{equation}\label{opt_hardy}
 \int_{\Omega}|\nabla \varphi|^2\dx -\gm\int_{\Omega} \frac{|\varphi|^2}{\gd_\Gw^2}\dx \geq \gl(\gm)\int_{\Omega}  \frac{|\varphi|^2}{|x|^2}\dx \qquad \forall \varphi\in C_0^\infty(\Omega).
\end{equation}
\end{Pro}

\begin{proof}
The fact that $\lambda(\mu)\geq0$ follows from $\sigma(\mu)\geq -(n-2)^2/4$, which has been proved in Theorem~\ref{thm_etreme}. Define
$$\psi(x)=|x|^{(2-n)/2}\phi_\mu\Big(\frac{x}{|x|}\Big).$$
Then, taking into account that
$$\Big(-\Delta_S-\sigma(\mu)-\frac{\mu}{\delta_\Omega^2}\Big)\phi_\mu=0\qquad \mbox{in } \Sigma,$$
and writing $P_\mu$ in spherical coordinates \eqref{eq_pm_spher}, it follows that $\psi$ is a positive solution of the equation
$$\left(P_\mu-\lambda(\mu)|x|^{-2}\right)u=0 \qquad \mbox{in } \Gw.$$
Thus, the operator $P_\mu-\lambda(\gm)|x|^{-2}$ is nonnegative in $\Omega$, and so \eqref{opt_hardy} holds by the AAP Theorem (Theorem~\ref{thm_AAP}).
\end{proof}

\begin{Rem}\label{supersol}

{\em

In the case $\mu<\mu_0$, the Hardy inequality \eqref{opt_hardy} can be obtained using the {\em supersolution construction} of \cite{DFP}: indeed, by Theorem~\ref{nthm_etreme}, the equation $P_\mu u=0$ has two linearly independent, positive solutions in $\Omega$, of the form
$$u_\pm(x)=|x|^{\gamma_\pm}\phi_{\mu}\Big(\frac{x}{|x|}\Big).$$
By the {\em supersolution construction} (\cite[Lemma~5.1]{DFP}), the positive function
$$u_{1/2}:=(u_+ u_-)^{1/2}=|x|^{(2-n)/2)}\phi_\mu\Big(\frac{x}{|x|}\Big)$$
is a solution of
$$\Bigg(P_\mu-\frac{\left|\nabla \left(u_+/u_-\right)\right|^2}{4 \left(u_+/u_-\right)^2}\Bigg)u=0 \qquad \mbox{in } \Gw.$$
It is easy to check that
$$\frac{\left|\nabla \left(u_+/u_-\right)\right|^2}{4 \left(u_+/u_-\right)^2}=\frac{\lambda(\mu)}{|x|^2},$$
and by the AAP~theorem, the Hardy inequality \eqref{opt_hardy} holds.
}
\end{Rem}

\begin{remark}\label{re_Fubini}{\em In the case $\mu\leq\mu_0$, the Hardy inequality \eqref{opt_hardy} can also be obtained using spherical coordinates, Fubini's theorem, and the well-known one-dimensional Hardy-inequality
 \be\label{Hardy_weighted_1dim}
  \int_0^\infty(v')^2t^{n-1}\dt
   \geq
    \Big(\frac{n-2}{2}\Big)^2\int_0^\infty v^2t^{n-3}\dt ,
 \ee
valid for all functions $v\in H^1(\mathbb{R}_+)$ that vanish at $\infty,$ one easily obtains \eqref{opt_hardy} for any $\gm\in\R$.

Indeed, suppose that $\varphi\in C_c^\infty(\Omega).$ Then we have that $\varphi_{\Sigma_r},$ the restriction of $\varphi$ on $\Sigma_r$, is in $C_c^\infty(\Sigma)$. Consequently, by the definition of $\sigma(\mu)$,  it follows that for all $\varphi\in C_c^\infty(\Omega)$ and each $r>0$ we have
 \be\nonumber
  \int_{\Sigma_r}|\nabla_\omega \varphi|^2 \dS_r
   -\mu\int_{\Sigma_r}\frac{\varphi^2}{\delta_{\Omega}^2(\omega)}\dS_r
    \geq
     \sigma(\mu)\int_{\Sigma_r}\varphi^2 \dS_r.
 \ee
Multiplying this by $r^{-2}$ and integrating in $\R_+$ with respect to $r$, we arrive at
 \be\nonumber
  \int_0^\infty\!\!\int_{\Sigma_r}\frac{|\nabla_\omega\varphi|^2}{r^2} \dS_r\dr
   -\mu\int_0^\infty\!\!\int_{\Sigma_r}\frac{\varphi^2}{r^2\delta_{\Omega}^2(\omega)}\dS_r\dr
    \geq
     \sigma(\mu)\int_0^\infty\!\!\int_{\Sigma_r}\frac{\varphi^2}{r^2} \dS_r\dr.
 \ee
Recall that in spherical coordinates we have $$|\nabla\varphi|^2=\frac{|\nabla_\omega\varphi|^2}{r^2}+\varphi_r^2,$$ and taking into account (\ref{eq_homo}), the last inequality is written as follows
 \be\nonumber
  \int_\Omega|\nabla \varphi|^2\dx
   -\mu\int_\Omega\frac{\varphi^2}{\delta_{\Omega}^2(x)}\dx
    \geq
     \sigma(\mu)\int_\Omega\frac{\varphi^2}{|x|^2}\dx
      +\int_{\Sigma}\int_0^\infty\varphi_r^2r^{n-1}\dr\dS,
 \ee
where we have used Fubini's theorem on the last term. Applying (\ref{Hardy_weighted_1dim}) in the inner integral of the last term and using Fubini's theorem again, we obtain (\ref{opt_hardy}) for any $\gm\in\R$.}
 \end{remark}

We now investigate the optimality of the Hardy inequality \eqref{opt_hardy} when $\mu<\mu_0$.
\begin{theorem}\label{thm_opt_hardy1}
Let $\Gw$ be a Lipschitz cone, and let $\mu<\mu_0$. Then $\gl(\gm)>0$. Furthermore the weight $W:=\gl(\gm)|x|^{-2}$ is an optimal Hardy weight for the operator $P_\gm$ in $\Gw$ in the following sense:
\begin{enumerate}
 \item The operator $P_\gm-\gl(\gm)|x|^{-2}$ is critical in $\Gw$, i.e., the Hardy inequality
 $$ \int_{\Omega}|\nabla \varphi|^2\dx -\gm\int_{\Omega} \frac{|\varphi|^2}{\gd_\Gw^2}\dx \geq \int_{\Omega} V(x)|\varphi|^2\dx \qquad \forall \varphi\in C_0^\infty(\Omega)$$
holds true for $V\geq W$ if and only if $V=W$. In particular,
$$\lambda_0\Big(P_\gm,\frac{1}{|x|^2},\Omega \Big)=\gl(\gm).$$
 \item The constant $\gl(\gm)$ is also the best constant for \eqref{opt_hardy} with test functions supported either in $\Gw_R$ or in $\Gw\setminus \overline{\Gw_R}$, where $\Gw_R$ is  a fixed truncated cone of the form \eqref{TC}. In particular,
$$\lambda_\infty\Big(P_\gm,\frac{1}{|x|^2},\Omega \Big)=\gl(\gm).$$
 \item The operator $P_\gm-\gl(\gm)|x|^{-2}$ is {\em null-critical} at $0$ and at infinity in the following sense: For any $R>0$  the (Agmon) ground state
of the operator $P_\gm-\gl(\gm)|x|^{-2}$  given by
$$v(x):=|x|^{(2-n)/2}\phi_\gm\Big(\frac{x}{|x|}\Big)$$
satisfies
$$\int_{\Gw_R} \left(|\nabla v|^2-\gm\frac{|v|^2}{\gd_\Gw^2}\right)\dx= \int_{\Omega\setminus \overline{\Gw_R}} \left(|\nabla v|^2-\gm\frac{|v|^2}{\gd_\Gw^2}\right)\dx=\infty.$$
 In particular, the variational problem
 \begin{align*}
 \inf_{\varphi\in \mathcal{D}^{1,\,2}_{P_\gm}(\Gw)}&\frac{\int_{\Omega}|\nabla \varphi|^2\dx-\gm\int_{\Omega} \frac{|\varphi|^2}{\gd_\Gw^2}\dx}{\displaystyle{\int_{\Omega}\frac{|\varphi|^2}{|x|^2}\dx}}\quad
 \end{align*}
 does not admit a minimizer.
\item The spectrum and the essential spectrum of the Friedrichs extension of the operator $W^{-1}P_\gm=\lambda(\mu)^{-1}|x|^2P_\mu$ on $L^2(\Gw, W\dx)$ are both equal to  $[1,\infty)$.
\end{enumerate}
\end{theorem}

\begin{Rem}
{\em

As is pointed out  in Remark~\ref{supersol}, if $\mu<\mu_0$, then the Hardy inequality \eqref{opt_hardy} can be obtained by applying the {\em supersolution construction} from \cite{DFP}. Thus, Theorem~\ref{thm_opt_hardy1} extends Theorem~\ref{temp_def} to the particular singular case, where $\Gw$ is a cone and $P_\mu$ is the Hardy operator (which is singular on $\partial \Omega$).

}
\end{Rem}

\begin{proof}[Proof of Theorem~\ref{thm_opt_hardy1}]
In light of our assumption that $\gm<\gm_0 \leq 1/4$, it follows the operator $P_\gm$ is subcritical in $\Gw$. Moreover, by Theorem~\ref{nthm_etreme}, $\sigma(\mu)>-(n-2)^2/4$, so $\lambda(\mu)>0$. For such a $\gm$, consider the operator $\mathcal{L}=\mathcal{L}_\gm$ on $\Sigma\subset \mathbb{S}^{n-1}$ defined by
$$\mathcal{L}=-\Delta_{S}-\frac{\mu}{\delta^2_{\Gw}}-\gs(\mu),$$ with the corresponding nonnegative quadratic form
$$q_{\mathcal{L}}(\psi)=\int_\Gs  \Bigg(|\nabla_\gw \psi|^2 -\gm \frac{|\psi|^2}{\gd_\Gw^2}-\gs(\mu)|\psi|^2\Bigg) \mathrm{d}S\qquad \mbox{where } \psi\in C_0^\infty(\Gs).$$
Notice that by Lemma~\ref{lem_pos_crit}, $\mathcal{L}$ is critical in $\Gs$ with the ground state $\phi_\gm\in L^2(\Sigma)$. We normalize $\phi_\mu$ so that $\int_\Sigma\phi_\mu^2\,\mathrm{d}S=1$.

On the other hand, it is well known that the operator
$$\mathcal{R}:=-\frac{\partial^2}{\partial r^2}-\frac{n-1}{r}\frac{\partial}{\partial r}-\frac{(n-2)^2}{4r^2}$$
is critical on $\R_+$, and $r^{(2-n)/2}$ is its ground state . Indeed, the corresponding quadratic form $q_{\mathcal{R}}$ of $\mathcal{R}$ (endowed with the measure $r^{n-1}\dr$) is given by
$$q_{\mathcal{R}}(u)=\int_0^\infty \left[(u')^2 -\frac{(n-2)^2}{4} \frac{u^2}{r^2}\right]r^{n-1}\dr\qquad u\in C_0^\infty(\R_+),$$
and gives rise to the critical operator $\mathcal{R}$ on $\R_+$.

Recall that in spherical coordinates $P_\mu-W$ has the following skew-product form:
$$P_\mu-W=\mathcal{R}\otimes \mathcal{I}_{\Gs}-\frac{\mathcal{I}_{\R_+}}{r^2}\otimes \mathcal{L}=\frac{\partial^2}{\partial r^2}-\frac{n-1}{r}\frac{\partial}{\partial r}-\frac{(n-2)^2}{4r^2}+\frac{1}{r^2}\mathcal{L},$$
where $\mathcal{I}_A$ is the identity operator on $A$. Consequently, it is natural to construct a null-sequence for $P_\mu-W$ of the product form $$\{\varphi_{k}(r,\gw)\}_{k= 1}^\infty=\{u_k(r)\phi_{k}(\gw)\}_{k= 1}^\infty$$ that  converges locally uniformly to $r^{(2-n)/2}\phi_\gm(\gw)$, and by Theorem~\ref{thm_crit}, this implies that the operator $P_\mu-W$ is critical and $r^{(2-n)/2}\phi_\gm(\gw)$ is its ground state.

Let $\{u_k(r)\}_{k= 1}^\infty$ be a null-sequence for the critical operator $\mathcal{R}$ on $\R_+$, converging locally uniformly to $r^{(2-n)/2}$. So,
$$q_{\mathcal{R}}(u_k)\to 0,\qquad \int_1^2 (u_k)^2\,r^{n-1}\dr=1.$$
On the other hand, let
$\{\phi_k(\gw)\}_{k= 1}^\infty$ be (up to the normalization constants) the sequence of ground states defined by \eqref{eq_pm_spher1} on $\Gs_k$, so that
 $$\int_{\Gs} \phi_k^2\,\mathrm{d}S=1, \quad  \mbox{and } q_{\mathcal{L}}(\phi_k)= \big(\gs_k(\gm)-\gs(\gm)\big)\int_{\Gs} \phi_k^2\,\mathrm{d}S \to 0. $$
Note that the normalization of $\phi_k$ is different from the one used in the proof of Theorem~\ref{thm_etreme}. Recall that the operator $\mathcal{L}_{\gm_0}=-\Delta_{S}-\mu_0\delta^{-2}_{\Gw}-\gs(\mu_0)$ is nonnegative on $\Sigma$. Therefore,
 \begin{equation}\label{eq_71}
    \frac{\gm\gs(\gm_0)}{\gm_0}\int_{\Gs} \phi_k^2\,\mathrm{d}S +\gm\int_{\Gs}\frac{|\phi_k|^2}{{\gd_\Gw^2}}\dS\leq \frac{\gm}{\gm_0} \int_{\Gs}|\nabla_\gw \phi_k|^2\dS.
 \end{equation}
On the other hand,
 \begin{equation}\label{opt_hardy5}
 \int_{\Gs}|\nabla_\gw \phi_k|^2\dS  =\gs_k \int_{\Gs}\phi_k^2\dS+\gm\int_{\Gs}\frac{\phi_k^2}{{\gd_\Gw^2}}\dS
\end{equation}
By \eqref{eq_71} and \eqref{opt_hardy5} we get
 \begin{equation}\label{eq_72}
    \Big(1-\frac{\gm}{\gm_0}\Big)\int_{\Gs}|\nabla_\gw \phi_k|^2\dS \leq \Big(\gs_k  -\frac{\gm\gs(\gm_0)}{\gm_0}\Big)\int_{\Gs} \phi_k^2\,\mathrm{d}S \leq \gs_1  -\frac{\gm\gs(\gm_0)}{\gm_0}
 \end{equation}
Since $\mu<\mu_0$, one gets that $\{\phi_k\}$ is bounded in $W^{1,2}_0(\Gs)$, and therefore (up to a subsequence), $\{\phi_k\}$ converges, in $L^2$ and locally uniformly to $\phi$, a positive solution of $\mathcal{L}u=0$ in $\Sigma$ with $\int_\Sigma \phi^2\,\mathrm{d}S=1$. Since $\mathcal{L}$ is critical in $\Gs$, $\phi=\phi_\mu$. Hence, by the Harnack inequality,
$$\int_{\Gs_1} \phi_k^2\,\mathrm{d}S\asymp 1,$$
and therefore $\{\phi_k\}$ is a null-sequence.

\medskip

We claim that there exists a subsequence $\{k_l\}\subset \N$, such that $\{u_l(r)\phi_{k_l}(\gw)\}$ is a null-sequence for the operator $P_\mu-W$ in $\Gw$ that converges locally uniformly to $r^{(2-n)/2}\phi_\gm(\gw)$.

Indeed, fix the pre-compact open set $B: =\{(r,\gw)\mid r\in (1,2),\,\omega\in \Sigma_1\}$. Note that for the quadratic form $Q$ of $P_\mu-W$ in $\Gw$, if $u=u(r)$ is compactly supported in $\R_+$ and $\psi=\psi(\gw)$ is compactly supported in $\Gs$, we have
$$Q(u(r)\psi(\gw))=q_{\mathcal{R}}(u)||\psi||_2^2+\left(\int_0^\infty u^2(r)r^{n-3}\dr\right)q_{\mathcal{L}}(\psi).$$
For each $k$, notice that by definition of a null-sequence, $u_k$ is compactly supported in $\R_+$. So, for $l\geq 1$, let $\{k_l\}_{l= 1}^\infty$ be a subsequence such that
$$q_{\mathcal{R}}(u_{l})||\phi_{k_{l}}||_2^2= q_{\mathcal{R}}(u_{l}) <\frac{1}{l}\,,$$
and
$$\left(\int_0^\infty u_{l}^2(r)r^{n-3}\dr\right)q_{\mathcal{L}}(\phi_{k_l})<\frac{1}{l}\,.$$
Thus, $\lim_{l\to\infty}Q(u_{l}(r)\phi_{k_l}(\gw))=0$.

On the other hand, $\{u_{l}(r)\phi_{k_l}\}$ converges uniformly in $B$ to the function  $r^{(2-n)/2}\phi_\gm(\gw)$, hence, $ \int_B \big(u_{l}(r)\phi_{k_l}(\gw)\big)^2\dx\asymp 1$.

\medskip

Therefore,  $\{u_{l}(r)\phi_{k_l}(\gw)\}_{l= 1}^\infty$ is indeed a null-sequence for $P_\mu-W$. It follows that $P_\mu-W$ is critical in $\Gw$ with the ground state $r^{(2-n)/2}\phi_\gm(\gw)$. Moreover, since $\mathcal{R}$ is null critical around $0$ and $\infty$ it follows $P_\mu-W$ is in fact null-critical around $0$ and $\infty$.



\medskip

Next we prove that the spectrum of $W^{-1}P_\gm$ is $[1,\infty)$. Let us keep our assumption that $\phi_\gm$ is normalized so that $||\phi_\gm||_2=1$. If $\xi\in \R$, then it easily checked (cf. \cite{DFP}) that
$$\left(\mathcal{R}-\frac{(n-2)^2\xi^2}{4|x|^2}\right)\left(r^{n-2}\right)^{i\xi-1/2}=0,$$
therefore,
\begin{equation}\label{eigen}
\left(P_\mu-\left(1+\frac{(n-2)^2}{4\lambda(\mu)}\xi^2\right)W\right)\left((r^{n-2})^{i\xi-1/2}\phi_\gm(\gw)\right)=0.
\end{equation}
Define the subspace $\mathcal{E}$ of $L^2(\Omega,W\dx)$ consisting of all functions of the form $u(r)\phi_\gm(\gw)$, where $u\in L^2 (\R_+,r^{n-1}\lambda(\mu)/r^2\dr)$. We are going to define a spectral representation of $W^{-1}P_\gm$ restricted to the subspace $\mathcal{E}$. Notice that the measure on $\mathcal{E}$ is $r^{n-1}\lambda(\mu)/(r^2)\dr\otimes \mathrm{d}S$, so that
$$\mathcal{E}=L^2\left(\R_+,r^{n-1}\frac{\lambda(\mu)}{r^2}\dr\right)\otimes \mathrm{span}\{\phi_\mu\}.$$
Recall that the classical Mellin transform is the unitary operator $\mathcal{M}: L^2(\R_+)\to L^2(\R)$ defined by
$$\mathcal{M}f(\xi)=\frac{1}{\sqrt{2\pi}}\int_0^\infty f(r)r^{i\xi-1/2}\dr.$$
Consider the composition $\mathcal{C}$ of the unitary operator
$$\mathcal{U}: L^2\left(\R_+,r^{n-1}\frac{\lambda(\mu)}{r^2}\dr\right) \to L^2(\R_+)$$ given by
$$f(r)\mapsto \sqrt{\frac{2\lambda(\mu)}{n-2}}f(r^{1/(n-2)}),$$
with the Mellin transform $\mathcal{M}$.
Define
$$\mathcal{T}:\mathcal{E}\mapsto L^2(\R); \quad \mathcal{T}(u(r)\phi_\mu(\gw))=(\mathcal{C}u)(\xi)=\Big(\mathcal{M}\big(\mathcal{U}(u)\big)\Big)(\xi).$$
So, $\mathcal{T}$ is a unitary operator. By \eqref{eigen}, the operator $\mathcal{T}(W^{-1}P_\mu)\mathcal{T}^{-1}$ is the multiplication by the real function $\left(1+(n-2)^2\xi^2/(4\lambda(\mu))\right)$ on $L^2(\R)$, with values in $[1,\infty)$. Therefore, the spectrum of $W^{-1}P_\mu$, restricted to $\mathcal{E}$ is $[1,\infty)$. So, the spectrum of $W^{-1}P_\mu$ on $L^2(\Omega,W\dx)$ contains $[1,\infty)$. But the Hardy inequality \eqref{opt_hardy} implies that the spectrum of $W^{-1}P_\mu$ must be included in $[1,\infty)$. Hence, the spectrum of $W^{-1}P_\mu$ on $L^2(\Omega,W\dx)$ is $[1,\infty)$.

\medskip

For $k\geq2$, define the subspace $\mathcal{E}_k$ (resp. $\mathcal{E}_{1/k}$) of $L^2(\Omega,W\dx)$ consisting of functions of the form $u(r)\phi(\gw)$, where $u\in L^2\big((k,\infty),r^{n-1}\lambda(\mu)/r^2\dr\big)$ (resp. $u\in L^2\big((0,1/k),r^{n-1}\lambda(\mu)/r^2\dr\big)$). Denote by $\mathcal{P}_k$ (resp. $\mathcal{P}_{1/k}$) the restriction of $P_\gm$ to $\mathcal{E}_k$ (resp. $\mathcal{E}_{1/k}$), with Dirichlet boundary conditions at $\{k\}\times\Sigma$ (resp. at $\{1/k\}\times \Sigma$). Notice that by symmetry considerations (under $x\mapsto x^{-1}$), the spectrum of $W^{-1}\mathcal{P}_k$ and the spectrum of $W^{-1}\mathcal{P}_{1/k}$ are equal. Moreover, by the fact that the essential spectrum is stable under compactly supported perturbations, and since the discrete spectrum of $W^{-1}P_\gm$ is empty, the spectrum of $W^{-1}P_\gm$ is equal to the union of the spectrum of $W^{-1}\mathcal{P}_k$, and of the spectrum of $W^{-1}\mathcal{P}_{1/k}$. Thus, the spectra of $W^{-1}\mathcal{P}_k$ and $W^{-1}\mathcal{P}_{1/k}$ are both equal to $[1,\infty)$.

Also, the best constant $C_0$ for the validity of the Hardy inequality
$$\int_{\mathcal{V}_0}\left(|\nabla \vgf|^2-\frac{\mu}{\delta^2_\Omega}\vgf^2\right)\dx \geq C_0\int_{\mathcal{V}_0} W\vgf^2\dx \qquad\forall \vgf\in C_0^\infty(\mathcal{V}_0),$$
in  $\mathcal{V}_0$, an arbitrarily small neighborhood of zero, is equal to the bottom of the essential spectrum of $W^{-1}\mathcal{P}_{1/k}$ (for any $k\geq 2$). Thus, it is equal to $1$. Similarly, using $W^{-1}\mathcal{P}_k$ instead, one concludes that the best constant $C_\infty$ for the validity of the Hardy inequality
$$\int_{\mathcal{V}_\infty}\left(|\nabla \vgf|^2-\frac{\mu}{\delta^2_\Omega}\vgf^2\right)\dx\geq C_\infty\int_{\mathcal{V}_\infty} W\vgf^2\dx\qquad\forall \vgf\in C_0^\infty(\mathcal{V}_\infty),$$
in  $\mathcal{V}_\infty$, an arbitrarily small neighborhood at infinity, is equal to $1$. This finishes the proof of Theorem~\ref{thm_opt_hardy1}.
\end{proof}
We now turn to the case $\mu=\mu_0$, for which we need to assume more regularity on $\Gs$.
\begin{theorem}\label{thm_opt_hardy2}
Assume that $\Gs\in C^2$.

1. If $\mu_0<1/4$, then $\lambda(\mu_0)=0$, and the operator $P_{\mu_0}$ is critical in $\Omega$, and null-critical around $0$ and $\infty$. In particular, the Hardy inequality
$$\int_\Omega |\nabla \vgf|^2\dx \geq \mu_0\int_\Omega \frac{\vgf^2}{\delta_\Omega^2}\dx \qquad\forall \vgf\in C_0^\infty(\Omega),$$
cannot be improved.

2. If $\mu_0=1/4$ and $\lambda(1/4)=0$, then the operator $P_{1/4}$ is critical in $\Omega$, and null-critical around $0$ and $\infty$. In particular, the Hardy inequality
$$\int_\Omega |\nabla \vgf|^2\dx\geq \frac{1}{4}\int_\Omega \frac{\vgf^2}{\delta_\Omega^2}\dx \qquad\forall \vgf\in C_0^\infty(\Omega),$$
cannot be improved.

3. If $\mu_0=1/4$ and $\lambda(1/4)>0$, then the weight $W_{1/4}:=\lambda(1/4)|x|^{-2}$ is optimal in the sense of Theorem~\ref{thm_opt_hardy1}. In particular, the Hardy inequality \eqref{opt_hardy} cannot be improved. Moreover, The spectrum and the essential spectrum of the Friedrichs extension of the operator $(W_{1/4})^{-1}P_{1/4}$ on $L^2(\Gw, W_{1/4}\dx)$ are both equal to  $[1,\infty)$.
\end{theorem}
\begin{proof}
Denote $W(x):=\lambda(\mu_0)|x|^{-2}$. Let us start by proving that in all cases, $P_{\mu_0}-W$ is critical. Recall that in spherical coordinates $P_{\mu_0}-W$ has the following skew-product form:
$$P_{\mu_0}-W=\mathcal{R}\otimes \mathcal{I}_{\Gs}-\frac{\mathcal{I}_{\R_+}}{r^2}\otimes \mathcal{L}=\frac{\partial^2}{\partial r^2}-\frac{n-1}{r}\frac{\partial}{\partial r}-\frac{(n-2)^2}{4r^2}+\frac{1}{r^2}\mathcal{L}_{\gm_0}\,.$$
So, as in  the first part of the proof of Theorem~\ref{thm_opt_hardy1}, it is natural to construct  a null-sequence for $P_{\mu_0}-W$ of the product form $$\{\varphi_{k}(r,\gw)\}_{k= 1}^\infty=\{u_k(r)\phi_{k}(\gw)\}_{k= 1}^\infty$$ that converges locally uniformly to $r^{(2-n)/2}\phi_{\mu_0}(\gw)$.

\medskip

As in the proof of Theorem~\ref{thm_opt_hardy1}, let $\{u_k(r)\}_{k= 1}^\infty$ be a null-sequence for the critical operator $\mathcal{R}$ on $\R_+$, converging locally uniformly to $r^{(2-n)/2}$. So,
$$q_{\mathcal{R}}(u_k)\to 0,\qquad \int_1^2 (u_k)^2\,r^{n-1}\dr=1.$$

However, the definition of $\{\phi_k\}$ differs from the one of Theorem~\ref{thm_opt_hardy1}.
Let us normalize $\phi_{\mu_0}$ so that $\int_\Sigma \phi_{\mu_0}^2\,\mathrm{d}S=1$ (by Lemma~\ref{lem_pos_crit}, $\gf_\gm\in L^2(\Gs)$). By lemmas~\ref{lem_pos_crit} and \ref{null-seq}, there exists a null-sequence $\{\phi_k\}$ for $\mathcal{L}_{\gm_0}$, converging locally uniformly  and in $L^2(\Gs)$ to $\phi_{\mu_0}$. Thus, normalizing $\phi_k$ so that
$$\int_\Sigma \phi_k^2\,\mathrm{d}S=1,$$
one has for $k$ large enough, by the Harnack inequality,
$$\int_{\Sigma_1}\phi_k^2\,\mathrm{d}S\asymp 1.$$
Let $B=\{(r,\gw)\mid r\in (1,2),\,\omega\in \Sigma_1\}$. We now choose the subsequence $\{k_l\}\subset \mathbb{N}$ as in the proof of Theorem~\ref{thm_opt_hardy1}: let $\{k_l\}_{l= 1}^\infty$ be a subsequence such that
$$q_{\mathcal{R}}(u_{l})||\phi_{k_{l}}||_2^2= q_{\mathcal{R}}(u_{l}) <\frac{1}{l}\,,$$
and
$$\left(\int_0^\infty u_{l}^2(r)r^{n-3}\dr\right)q_{\mathcal{L}}(\phi_{k_l})<\frac{1}{l}\,.$$
The same computation made in the proof of Theorem~\ref{thm_opt_hardy1} shows that
$$\lim_{l\to\infty}Q(u_{l}(r)\phi_{k_l}(\gw))=0, \quad \mbox{ and } \int_B \big(u_{l}(r)\phi_{k_l}(\gw)\big)^2\dx\asymp 1,$$
so that $\{u_{l}(r)\phi_{k_l}(\gw)\}_{l= 1}^\infty$ is indeed a null-sequence for $P_\mu-W$. It follows that $P_\mu-W$ is critical in $\Gw$ with a ground state $r^{(2-n)/2}\phi_\gm(\gw)$. Moreover, since $\mathcal{R}$ is null critical around $0$ and $\infty$ it follows $P_\mu-W$ is in fact null-critical around $0$ and $\infty$.

\medskip

1. Assume now that $\mu_0<1/4$. By the first part of the proof, the operator $P_\mu-\lambda(\mu)|x|^{-2}$ is critical, and null-critical around $0$ and $\infty$. By Lemma~\ref{lem_pos_crit}, $\sigma(\mu_0)=-(n-2)^2/4$, so $\lambda(\mu_0)=0$. It follows that $P_{\mu_0}$ is critical, and null-critical around $0$ and $\infty$.

\medskip

2. Suppose that $\mu_0=1/4$, and $\lambda(1/4)=0$. Then by the first part of the proof, the operator $P_{1/4}=P_{1/4}-\lambda(1/4)|x|^{-2}$ is critical, and null-critical around $0$ and $\infty$.

\medskip

3. Assume that $\mu_0=1/4$, and $\lambda(1/4)>0$. Then following the proof of Theorem~\ref{thm_opt_hardy1}, one concludes that $W$ is an optimal weight for $P_{1/4}$.
\end{proof}
In the particular case of the half-space we can compute the constants appearing in theorems~\ref{thm_opt_hardy1} and \ref{thm_opt_hardy2}.
\begin{example}[see {\cite[Example~11.9]{DFP}} and \cite{FTT}]\label{ex2}
{\em
Let $\Gw=\Real^n_+$,  $\mu\leq \gm_0=1/4$ and consider the subcritical operator $P_\mu:=-\Gd -\mu x_1^{-2}$ in $\Gw$.  Let $\ga_+$ be the largest root of the equation
$\ga(1-\ga)=\gm$, and let $$\gh(\gm):=n-1+\sqrt{1-4\gm}=n-2+2\ga_+.$$
Then
$$v_0(x):=x_1^{\ga_+}, \qquad v_1(x):=x_1^{\ga_+}|x|^{-\gh(\gm)}$$ are two positive solutions of the equation
$P_\gm u=0$ in $\Gw$ that vanish on $\partial \Gw\setminus\{0\}$.

Therefore, $\gl(\gm)=\gh^2(\gm)/4$, and for $\gm\leq \gm_0=1/4$ we have  the following optimal Hardy inequality
$$\int_{\R^n_+}\!\!|\nabla \varphi|^2\dx-\mu\int_{\R^n_+}\!\!\frac{\varphi^2}{x_1^2}\dx \!\geq\!\frac{\eta^2(\mu)}{4}\!\!\!\int_{\R^n_+}\!\!\frac{\varphi^2}{|x|^2}\dx \quad \forall \varphi\!\in \!C_0^\infty(\R^n_+).$$
In particular, the operator $-\Gd -\mu x_1^{-2}-\gl(\gm)|x|^{-2}$ is critical in $\Real^n_+$ with the ground state $\psi(x):=x_1^{\ga_+}|x|^{-\gh(\gm)/2}$. Note that for $\mu=0$ we obtain the well known (optimal) Hardy inequality (see \cite{Nz})
$$\int_{\Real^n_+}|\nabla \varphi|^2\dx\geq \frac{n^2}{4}\int_{\Real^n_+}\frac{\varphi^2}{|x|^2}\dx \qquad \forall\varphi\in C_0^\infty(\Real^n_+),$$
while for $\gm=\gm_0=1/4$ we obtain the optimal double Hardy inequality (see \cite{FTT})
\begin{equation}\label{psar}
\int_{\Real^n_+}|\nabla \varphi|^2\dx- \frac{1}{4}\int_{\Real^n_+}\frac{1}{x_1^2}\varphi^2\dx \geq\frac{(n-1)^2}{4}\int_{\Real^n_+}\frac{\varphi^2}{|x|^2}\dx \quad \forall\varphi\in C_0^\infty(\Real^n_+).
\end{equation}
 }
\end{example}
It turns out that in the weakly mean convex case, $\lambda(1/4)$ is always positive.
\begin{Pro}\label{pro_moztkin}
Assume that $\Sigma\in C^2$ and $\Gw$ is  weakly mean convex. Then $\lambda(1/4)>0$.
\end{Pro}
\begin{proof}
Since $\Gw$ is  weakly mean convex ( i.e., $-\Gd\gd_\Gw\geq 0$ in $\Gw$), it follows that $\gd_\Gw^{1/2}$ is a positive supersolution of $P_{1/4}u=0$ in $\Gw$. We proceed by contradiction: assume that $\lambda(1/4)=0$. Then by Theorem~\ref{thm_opt_hardy2} the operator $P_{1/4}$ is critical and therefore $\gd_\Gw^{1/2}$ is a positive solution of $P_{1/4}u=0$ in $\Gw$. Thus, necessarily $-\Gd\gd_\Gw=0$ in the sense of distributions. Since $\gd_\Gw\in W^{1,2}_{\rm loc}(\Gw)$ (or directly by Weyl's lemma) we have that $\gd_\Gw$ is harmonic and in particular $\gd_\Gw\in C^\infty(\Gw).$ This means that the singular set of $\gd_\Gw,$
 \begin{align*}
  {\rm Sing}(\gd_\Gw) & :=
   \{x\in\Gw\mid \gd_\Gw(x)\mbox{ is achieved by more than one boundary points}\}
    \\ \nonumber & \;= \{x\in\Gw\mid \gd_\Gw\mbox{ is not differentiable}\},
 \end{align*}
(see for example \cite[Theorem 3.3]{EvH}) is empty. In light of Motzkin theorem \cite[Theorem 1.2.4]{Sch},   $\R^n\setminus\Gw$ is convex. Since $0$ is on the boundary of $\R^n\setminus \Gw$, by considering a supporting hyperplane of $\R^n\setminus \Gw$ at $0$, we find that necessarily $\R^n\setminus\Gw$ is included in a half-space. This implies that $\Gs$ contains a half-sphere. If this half-sphere is strictly contained in $\Gs$, then $K:=\R^n\setminus\Gw$ is a closed convex cone not containing a line (i.e., $K$ is {\em pointed}). Hence, its dual cone $K^*$, and thus its polar cone $K^o=-K^*\subset\Gw$ has nonempty interior (see for instance \cite[page~53]{BV}). Clearly, $\gd_\Gw(x)=|x|$ whenever $x\in K^o$, but this contradicts the harmonicity of  $\gd_\Gw$ in $\Gw$.

Hence, $\Gs$ is precisely a half-sphere, and thus $\Gw$ is a half-space. But by Example~\ref{ex2}, in the half-space $\{x_1>0\}$ we have $\lambda(1/4)=(n-1)^2/4>0$, and we arrived at a contradiction.
\end{proof}
Assume that $\Gw$ is a domain admitting a supporting hyperplane $H$ at zero. Without loss of generality, we may assume that $H=\partial\R^n_+$. Recall that in this case
$\gl_0(-\Gd,\gd_\Gw^{-2},\Gw)\leq 1/4$ \cite[Theorem~5]{MMP}. Also, $\gd_\Gw\leq \gd_H$ in $\Gw$. Consequently, for appropriate test functions $\vgf_\vge$ supported in a relative small neighborhood of the origin in $\Gw$ we have that for $0\leq \gm\leq 1/4$ the corresponding Rayleigh-Ritz quotients satisfy the inequality
\begin{equation*}
\frac{\int_{\Gw}  \left(|\nabla\vgf_\vge|^2 -  \gm\frac{|\vgf_\vge|^2}{\gd_\Gw^2}\right) \dx}{\int_{\Gw}  \frac{|\vgf_\vge|^2}{|x|^2} \dx}
 \leq
  \frac{\int_{H} \left (\nabla\vgf_\vge|^2 -  \gm\frac{|\vgf_\vge|^2}{\gd_H^2}\right) \dx}{\int_{H}  \frac{|\vgf_\vge|^2}{|x|^2} \dx}
   =
    \frac{\left(n-1+\!\sqrt{1-4\gm}\right)^2}{4} +o(1),
\end{equation*}
where $o(1)\to 0$ as $\vge\to 0$.  Thus, Example~\ref{ex2} implies
\begin{Cor}\label{cor1}
Suppose that a domain $\Gw$ admits a supporting hyperplane at zero, and let $P_\gm=-\Gd-\gm\gd_{\Gw}^{-2}$, where $0\leq \gm\leq 1/4$. Then
$$\gl_0(P_\gm,|x|^{-2},\Gw) \leq \frac{\left(n-1+\!\sqrt{1-4\gm}\right)^2}{4}\,.$$
\end{Cor}
\section{On the optimality of an inequality by Filippas, Tertikas and Tidblom}\label{sec_FTT}
In the present section we generalize examples~\ref{ex1} and \ref{ex2} concerning the half-space $\R^n_+$. We consider the following family of Hardy inequalities in $\R^n_+$, obtained by S. Filippas, A. Tertikas and J. Tidblom \cite{FTT}:
\begin{equation}\label{H-FTT}
\int_{\R^n_+}\!\!\!|\nabla \vgf|^2\,\mathrm{d}x\!\geq\! \int_{\R^n_+}\!\!\!\Big(\frac{\beta_1}{x_1^2}\!+\!\frac{\beta_2}{x_1^2\!+\!x_2^2}\!+\!\ldots+\frac{\beta_n}{x_1^2\!+\!\ldots\!+\!x_n^2}\Big)\,\!\!\vgf^2\,\mathrm{d}x \;\; \;\;\forall \vgf\!\in\! C_0^\infty(\R^n_+).
\end{equation}
According to \cite[Theorem A]{FTT}, the Hardy inequality \eqref{H-FTT} holds if and only if the $\beta_i$'s are of the following form:
\begin{equation}\label{beta}
\beta_1=-\alpha_1^2+\frac{1}{4},\quad \beta_i=-\alpha_i^2+\Big(\alpha_{i-1}-\frac{1}{2}\Big)^2\qquad i=2,\ldots,n,
\end{equation}
where the $\alpha_i$'s are arbitrary real numbers. Without loss of generality, we can --and will-- assume that all $\alpha_i$'s in \eqref{beta} are nonpositive . Denote
$$V(\beta_1,\ldots,\beta_j)=\Big(\frac{\beta_1}{x_1^2}+\frac{\beta_2}{x_1^2+x_2^2}+\ldots+\frac{\beta_j}{x_1^2+\ldots+x_j^2}\Big)\qquad j=1,\ldots, n .$$
Let $2^*=2n/(n-2)$ be the Sobolev exponent. In \cite[Theorem B]{FTT}, it is shown that \eqref{H-FTT} can be improved by adding to the right-hand side a Sobolev term of the form $C(\int_{\R^n_+}|\vgf|^{2^*}\dx)^{2/2^*}$  if and only if $\alpha_n<0$. Notice that $\beta_1,\ldots,\beta_{n-1}$ being fixed, taking $\alpha_n=0$ corresponds to taking the greatest $\beta_n$ possible in \eqref{beta}.

Our aim in this section is to show that when $\alpha_n=0$, not only one cannot add a Sobolev term, but in fact one cannot even add any term of the form $\int_{\R^n_+}W\vgf^2\dx$, $W\gneqq0$, to the right hand side of \eqref{H-FTT}. In other words, if $\alpha_n=0$, the operator $-\Delta-V(\beta_1,\ldots,\beta_n)$ is {\em critical} in $\R^n_+$. This implies in particular (see \cite{PT2}) that \eqref{H-FTT} cannot be improved by adding to the right-hand side any weighted Sobolev term of the form $C(\int_{\R^n_+}\rho|\vgf|^{2^*}\dx)^{2/2^*}$, where $\rho\gneqq0$; an improvement of the result obtained in \cite{FTT}.
\begin{theorem}\label{FTT-crit}
Consider the Hardy inequality \eqref{H-FTT}, where the $\beta_i$'s are defined in term of nonpositive $\alpha_i$'s by \eqref{beta}. Assume that $\alpha_n=0$, and that $\alpha_1,\ldots,\alpha_{n-1}$ are either all distinct, or all negative. Then the operator $P:=-\Delta-V(\beta_1,\ldots,\beta_n)$ is critical in $\R^n_+$, i.e., the Hardy inequality \eqref{H-FTT} cannot be improved. Furthermore, the weight $\beta_n|x|^{-2}$ is an optimal weight for the subcritical operator $-\Delta-V(\beta_1,\ldots,\beta_{n-1})$ in $\R^n_+$.
\end{theorem}
\begin{proof}
Denote $X_k(x):=(x_1,\ldots,x_k,0,\ldots,0)$. Let $(\beta_i)_{i=1}^n$ satisfy \eqref{beta}, and define
$$\psi(x):=|X_1|^{-\gamma_1}|X_2|^{-\gamma_2}\ldots|X_n|^{-\gamma_n},$$
where $\gamma_i$ are defined by
$$\gamma_1=\alpha_1-\frac{1}{2},\quad\gamma_i=\alpha_i-\alpha_{i-1}+\frac{1}{2} \qquad i=2,\ldots,n.$$
Then,
$$\beta_1=-\gamma_1(1+\gamma_1),\quad \beta_i=-\gamma_i\Big(2-i+\gamma_i+2\sum_{k=1}^{i-1}\gamma_k\Big)\qquad i=2,\ldots,n,$$
and according to equality (2.3) in \cite{FTT},
$$-\frac{\Delta\psi}{\psi}=V(\beta_1,\ldots,\beta_n).$$
Hence, $\psi$ is a positive solution of the equation $Pu=0$ in $\R^n_+$. By the AAP Theorem, this implies the validity of \eqref{H-FTT}.

For $x=(x_1,\ldots,x_n)\in \R^n_+\setminus\{0\}$, denote
$$r=|x|, \quad \omega=\frac{x}{|x|}\,, \quad \omega_i=\frac{x_i}{r}\quad 1\leq i\leq n.$$
Notice that $\omega\in  \mathbb{S}_+:=\mathbb{S}^{n-1}\cap \{x_1>0\}$. Since $\ga_n=0$ we have
$$\psi(x)=\phi(\omega)r^{-\sum_{i=1}^n\gamma_i}=\phi(\omega)r^{(2-n)/2},$$
where
$$\phi(\gw):=\psi|_{\mathbb{S}_+}= \omega_1^{-\gamma_1}(\omega^2_1+\omega^2_2)^{-\gamma_2/2}\cdots (\omega_1^2+\cdots+\omega_n^2)^{-\gamma_n/2}.$$
Define
$$W(\omega):=\frac{\beta_1}{\omega_1^2}+\ldots+\frac{\beta_{n-1}}{\omega_1^2+\ldots+\omega_{n-1}^2}\,,$$
and let
$$\mathcal{L}:=-\Delta_{\mathbb{S}^{n-1}}-W(\omega)-\beta_n+\frac{(n-2)^2}{4}\,,\;\;\;\mbox{ and }\;\;\;\mathcal{R}:=-\frac{\partial^2}{\partial r^2}-\frac{n-1}{r}\frac{\partial}{\partial r}-\frac{(n-2)^2}{4r^2}\,.$$
Then, in spherical coordinates, $P$ has the skew-product form
$$P=\mathcal{R}+\frac{1}{r^2}\mathcal{L}.$$
Recall that $\mathcal{R}$ is critical on $(0,\infty)$, and its ground state is $r^{(2-n)/2}$.
\begin{Lem}\label{criticality}
The operator $\mathcal{L}$ is critical on $\mathbb{S}_+$, with ground state $\phi\in L^2(\mathbb{S}_+)$.
\end{Lem}
Once Lemma~\ref{criticality} is proved, the rest of the proof of Theorem~\ref{FTT-crit} follows along the lines of the proof of Theorem~\ref{thm_opt_hardy1}.
\end{proof}
\begin{proof}[Proof of Lemma~\ref{criticality}] We have
$$P\psi=0=\phi\mathcal{R}r^{(2-n)/2}+r^{-(n+2)/2}\mathcal{L}\phi.$$
Since
$$\mathcal{R}r^{-(n-2)/2}=0\qquad \mbox{in } \R_+,$$
one concludes that
$$\mathcal{L}\phi=0  \qquad \mbox{in } \mathbb{S}_+.$$
For $x\in \mathbb{S}_+$, let $\gr$ be the spherical distance function to $\partial \mathbb{S}_+=\{\gw\in \mathbb{S}_+ \mid \omega_1=0\}$, the boundary of  $\mathbb{S}_+$. Let $\mathrm{d}S$ be the Riemannian measure on $\mathbb{S}_+$.  We claim that
\begin{equation}\label{integral}
\int_{\mathbb{S}_+\cap \{\gr\leq \frac{1}{2}\}}\Big(\frac{\phi(\omega)}{\gr\log(\gr)}\Big)^2\dS<\infty.
\end{equation}
Clearly, \eqref{integral} implies that $\phi\in L^2(\mathbb{S}_+)$, and moreover, by Lemma~\ref{null-seq}, \eqref{integral} implies that $\mathcal{L}$ is critical with the ground state $\phi$. In fact, since $\phi$ is smooth in the interior of $\mathbb{S}_+$, and
$$\gr(\gw)\sim \omega_1(\gw)\, \mbox{ as } \,\gw\in \mathbb{S}_+,\, \mbox{ and } \gr(\gw)\to 0,$$
\eqref{integral} is equivalent to
\begin{equation}\label{int}
\int_{\mathbb{S}_+\cap\{\omega_1\leq \frac{1}{2}\}}\Big(\frac{\phi(\omega)}{\omega_1\log(\omega_1)}\Big)^2\dS<\infty.
\end{equation}
For $i=1,\ldots,n-1$, define
$$\mathcal{E}_i=\{\omega\in \mathbb{S}_+ \mid \omega_1\leq \varepsilon,\ldots,\omega_i^2\leq \varepsilon,\omega_{i+1}^2 > \varepsilon\}.$$
Then, all the $\mathcal{E}_i$ are disjoint, and if $\varepsilon<1/n$, one can write the $\varepsilon$-neighborhood $\mathbb{S}_+\cap\{\omega_1\leq \varepsilon\}$ of $\partial \mathbb{S}_+$ as the disjoint union:
$$\mathbb{S}_+\cap\{\omega_1\leq \varepsilon\}=\mathcal{E}_1\cup \ldots\cup \mathcal{E}_{n-1}.$$
Notice that on $\mathcal{E}_i$,
$$\phi(\gw)\leq C_\varepsilon \omega_1^{-\gamma_1}(\omega^2_1+\omega^2_2)^{-\gamma_2/2}\cdots (\omega_1^2+\cdots+\omega_i^2)^{-\gamma_i/2}.$$
Hence,
$$\int_{\mathcal{E}_i}\Big(\frac{\phi(\omega)}{\omega_1\log(\omega_1)}\Big)^2\,\mathrm{d}S\leq C_\vge\int_{\mathcal{E}_i}\log^{-2}(\omega_1)\omega_1^{-2}\omega_1^{-2\gamma_1}\cdots (\omega_1^2+\cdots+\omega_i^2)^{-\gamma_i}\,\mathrm{d}S.$$
If $\varepsilon$ is small enough, then on $\mathcal{E}_i$,
$$\dS\simeq \,\mathrm{d}\omega_1\otimes\ldots\otimes \mathrm{d}\omega_i \otimes \mathrm{d}\nu(\omega_1,\ldots,\omega_i),$$
where $\mathrm{d}\nu(\omega_1,\ldots,\omega_i)$ is the standard Hausdorff measure on the $n-i-1$-sphere $\omega_{i+1}^2+\cdots+\omega_n^2=\sigma^2$, with  $\sigma^2=1-(\omega_1^2+\cdots+\omega_i^2)$. Thus,
\be\label{int2}
 \int_{\mathcal{E}_i}\Big(\frac{\phi(\omega)}{\omega_1\log(\omega_1)}\Big)^2\mathrm{d}S
  \leq
   \tilde{C}_\varepsilon \int_{[0,\varepsilon]^{i}}\log^{-2}(\omega_1)\omega_1^{-2}\omega_1^{-2\gamma_1}\cdots (\omega_1^2+\cdots+\omega_i^2)^{-\gamma_i}\, \mathrm{d}\omega_1\ldots \mathrm{d}\omega_i.
\ee
For $\lambda_1,\ldots,\lambda_i$ real numbers and $k$ integer, define
$$
I_i(\lambda_1,\ldots,\lambda_i,k):=\!\!\!\int_{[0,\varepsilon]^{i}}\!\!\!\!\!\!\log^{-2}(\omega_1)\omega_1^{-2}\omega_1^{-2\gl_1}\cdots (\omega_1^2+\cdots+\omega_i^2)^{-\gl_i} |\log^k(\omega_1^2+\cdots+\omega_i^2)|\, \mathrm{d}\omega_1\ldots \mathrm{d}\omega_i.
$$
One has the elementary fact:
\begin{equation}\label{calc}
I_i(\lambda_1,\ldots,\lambda_i,k)\leq C_\varepsilon \left\{\begin{array}{lcl}
I_{i-1}(\lambda_1,\ldots,\lambda_{i-2},\lambda_{i-1}+\lambda_i-1/2,k),&&\lambda_i>1/2,\\\\
I_{i-1}(\lambda_1,\ldots,\lambda_{i-1},k),&&\lambda_i<1/2,\\\\
I_{i-1}(\lambda_1,\ldots,\lambda_{i-1},k+1),&&\lambda_i=1/2.
\end{array}\right.
\end{equation}

\medskip\noindent
{\em Case 1: assume that the $\alpha_k$'s, $k=1,\ldots,n-1$, are all distinct}. Then, for every $2\leq j\leq k\leq i$,
$$\gamma_j+\sum_{l=j+1}^k\Big(\gamma_l-\frac{1}{2}\Big)=\alpha_k-\alpha_{j-1}+\frac{1}{2}\neq \frac{1}{2}.$$
Moreover,
\begin{equation}\label{eq_k}
    -2-2\gamma_1-2\sum_{j=2}^k \Big(\gamma_j-\frac{1}{2}\Big)=-2-2\alpha_k-(k-2)+(k-1)=-2\alpha_k-1.
\end{equation}
Thus, by using \eqref{calc} $i$-times in \eqref{int2}, and \eqref{eq_k}, one gets
\bea\nonumber
 \int_{\mathcal{E}_i}\Big(\frac{\phi(\omega)}{\omega_1\log(\omega_1)}\Big)^2\mathrm{d}S
  & \leq &
   C\sum_{k=1}^i\int_0^\varepsilon \log(\omega_1)^{-2}\omega_1^{-2-2\gamma_1-2\sum_{j=2}^k (\gamma_j-1/2)}\mathrm{d}\omega_1
    \\ \nonumber & \leq &
     C\sum_{k=1}^i\int_0^\varepsilon\log(\omega_1)^{-2}\omega_1^{-2\alpha_k-1}\mathrm{d}\omega_1,
\eea
where by convention the sum $\sum_{j=2}^k$ is zero when $k=1$. By hypothesis, $\alpha_k\leq 0$, therefore $\log(\omega_1)^{-2}\omega_1^{-2\alpha_k-1}$ is integrable at zero, and thus one concludes the validity of \eqref{integral}.\hspace{\fill}
\medskip

\noindent{\em Case 2: assume that $\alpha_k<0$, for all $k=1,\ldots,n-1$.} Then, by using \eqref{calc} $i$-times in \eqref{int2}, and \eqref{eq_k}, one gets
\bea\nonumber
 \int_{\mathcal{E}_i}\Big(\frac{\phi(\omega)}{\omega_1\log(\omega_1)}\Big)^2\mathrm{d}S
  & \leq &
   C\sum_{k=1}^i\int_0^\varepsilon |\log^{n(k)}(\omega_1)|\omega_1^{-2-2\gamma_1-2\sum_{j=2}^k (\gamma_j-1/2)}\mathrm{d}\omega_1
    \\ \nonumber & \leq &
     C\sum_{k=1}^i\int_0^\varepsilon|\log^{n(k)}(\omega_1)|\omega_1^{-2\alpha_k-1}\mathrm{d}\omega_1,
\eea
where $n(k)$ is an integer. Since $\alpha_k<0$, the function $|\log^{n(k)}(\omega_1)|\omega_1^{-2\alpha_k-1}$ is integrable at zero, and therefore \eqref{integral} holds.
\end{proof}

\begin{Rem}
{\em We believe that Theorem~\ref{thm_opt_hardy2} should hold in the general case, without any extra assumption on $\alpha_1,\ldots,\alpha_{n-1}$. We leave this question for a future investigation.
}
\end{Rem}
\section{A differential inequality}\label{sec_DE}
Throughout the present section, $\Gw$ denotes a domain in $\R^n$ such that $0\in\partial \Gw$, and $P_\gm=-\Gd-\gm\gd_\Gw^{-2}\,$. Our aim is to obtain a Hardy-type inequality with the best constant for the (nonnegative) operator $P_\gm$ in $\Gw$, assuming that $\gd_\Gw$ satisfies the linear differential inequality
\begin{equation}\label{de}
-\Gd \gd_\Gw +\frac{n-1+\sqrt{1-4\mu}}{|x|^2}\big(x\cdot\nabla \gd_\Gw-\gd_\Gw\big) \geq 0\qquad \mbox{in }\Gw.
\end{equation}
The above differential inequality certainly holds true for any $\gm\leq 1/4$ if $\Gw$ is a weakly mean convex cone (see Definition~\ref{def_mean_conv}); it also holds for $\mu=1/4$ if $\Gw$ is a ball touching the origin (see Remark~\ref{rem3}).

For $\gm=1/4$, \eqref{de} is equivalent to the differential inequality
$$-|x|^{n-1}\div\left(|x|^{1-n}\nabla \gd_\Gw\right)- \frac{n-1}{|x|^{2}}\,\gd_\Gw \geq 0 \qquad \mbox{in }\Gw.$$
It is worth mentioning here that in \cite[Theorem~3.2]{FMT} S.~Filippas, L.~Moschini, and A.~Tertikas obtain improved Hardy inequality under the assumption that $\Gw$ is a {\em bounded} domain such that $0\in \Gw$, and $\gd_\Gw$ satisfies  the differential inequality
$$ -\div\left(|x|^{2-n}\nabla \gd_\Gw\right) \geq 0 \qquad \mbox{in }\Gw,  $$
while K.T.~Gkikas in \cite{G} proves the Hardy inequality in an {\em exterior} domain $\Gw$ such that $0\in \R^n\setminus \bar{\Gw}$, and $\gd_\Gw$ satisfies  the differential inequality
$$ -\div\left(|x|^{1-n}\nabla \gd_\Gw\right) \geq 0 \qquad \mbox{in }\Gw.  $$

Let
\begin{equation}\label{eq_gh}
\gh(\gm) :=  n-1+\sqrt{1-4\gm}.
\end{equation}
Recall that for $\Gw=\R^n_+$, we obtained in Example~\ref{ex2} that $\gl_0(P_\gm,|x|^{-2},\Gw)=\gh^2(\gm)/4$.
The following theorem shows that if $\Gw$ is a domain such that $\gd_\Gw$  is a positive supersolution of a certain second-order linear elliptic equation, then
$\gl_0(P_\gm,|x|^{-2},\Gw)\geq \gh^2(\gm)/4$.
\begin{theorem}
Let $\Gw$ be a domain in $\R^n$ such that $0\in\partial \Gw$. Fix $\gm\leq 1/4$, and let $\gh(\gm)$ be as in \eqref{eq_gh}. Suppose that $\gd_\Gw$ satisfies the following differential inequality
\begin{equation}\label{eq_diff_ineq}
-\Gd \gd_\Gw +\frac{\eta(\mu)}{|x|^2}\big(x\cdot\nabla \gd_\Gw-\gd_\Gw\big) \geq 0 \qquad \mbox{ in } \Gw
\end{equation}
in the sense of distributions. Then the following improved Hardy inequality holds
\begin{equation}\label{opt_hardy4}
 \int_{\Omega}|\nabla \varphi|^2\dx
  -\gm\int_{\Omega} \frac{|\varphi|^2}{\gd_\Gw^2}\dx
   \geq \frac{\gh^2(\gm)}{4}\int_{\Omega} \frac{|\varphi|^2}{|x|^2}\dx
    \qquad \forall \varphi\in C_0^\infty(\Omega).
\end{equation}
 Assume further that $\Gw$ admits a supporting hyperplane at zero and $\gm\geq 0$, then $$ \gl_0(P_\gm,|x|^{-2},\Gw)=\frac{\gh^2(\gm)}{4}\,.$$
\end{theorem}
\begin{proof}

As in Example \ref{ex2}, we write $\ga_+$ for the largest root of the equation $\ga(1-\ga)=\gm$, and $\psi:=\delta_\Gw^{\alpha_+}|x|^{-\eta(\mu)/2}$. We will show that $\psi$  is a supersolution of the equation
$$\left(P_\mu-(\eta(\mu)/2)^2|x|^{-2}\right)u=0\qquad \mbox{in } \Gw,$$ and then (\ref{opt_hardy4}) follows from the AAP theorem (Theorem \ref{thm_AAP}). By direct computations we get
 \begin{align*}\nonumber
  & \Big(P_\mu-\frac{\eta^2(\mu)}{4|x|^2}\Big)\psi
   \\[2mm] & =
    \ga_+\Big(-\Gd\gd_\Gw+\frac{\eta(\mu)}{|x|^2}x\cdot\nabla\gd_\Gw\Big)\gd_\Gw^{\ga{_+}-1}|x|^{-\eta(\mu)/2}
     +
      \frac{\eta(\mu)}{2}\big(n-2-\eta(\mu)\big)\gd_\Gw^{\ga_+}|x|^{-\eta(\mu)/2-2}
       \\[2mm] & =
        \ga_+\Big(-\Gd\gd_\Gw+\frac{\eta(\mu)}{|x|^2}\big(x\cdot\nabla\gd_\Gw-\gd_\Gw\big)\Big) \geq 0,
\end{align*}
where for the second equality we have used the fact that $n-2-\gh(\gm)=-2\ga_+,$ which follows from our choice of $\ga_+$.

Assume that $\Gw$ is a domain admitting a supporting hyperplane $H$ at zero. Without loss of generality, we may assume that $H=\partial\R^n_+$.
 Then by Corollary~\ref{cor1} we have that $\gl_0(P_\gm,|x|^{-2},\Gw)\leq\gh^2(\gm)/4$.
Thus, $\gl_0(P_\gm,|x|^{-2},\Gw)=\gh^2(\gm)/4.$
\end{proof}

\begin{remark}\label{rem3}{\em
1. By \eqref{eq_Eu_homo1}, inequality \eqref{eq_diff_ineq} holds true  for any $\gm\leq 1/4$ if $\Gw$ is a weakly mean convex cone.

We claim that \eqref{eq_diff_ineq} holds true also for $\gm= 1/4$ in any ball $B$ with $0\in \partial B$, and consequently, the Hardy inequality \eqref{opt_hardy4} is valid in this case.

Indeed, let $B=B_R(x_0)$ be an open ball in $\Rn$ centered at $x_0,$ such that $|x_0|=R$. Then for $x\in B$ we have $\gd_{B}(x)=R-|x_0-x|$, and simple computations show that for any $x\in B\setminus\{x_0\}$
 \be\nonumber
  \nabla\gd_{B}(x)
   =
    \frac{x_0-x}{|x_0-x|}
     \hspace{1em}\mbox{ and }\hspace{1em}
      -\Delta\gd_{B}(x)
       =
        \frac{n-1}{|x_0-x|}\,.
 \ee
Thus, for (\ref{eq_diff_ineq}) to be true it is enough that for any $x\in B\setminus\{x_0\}$ we have
  $$-\Gd \gd_B +\frac{\eta(\mu)}{|x|^2}\big(x\cdot\nabla \gd_B-\gd_B\big)=\frac{n-1}{|x_0-x|}+\frac{n-1}{|x|^2}\left(x\cdot \frac{(x_0-x)}{|x_0-x|}-R+|x_0-x|\right)\geq 0.$$
    After some cancelations this is equivalent to
 \be\label{eq_enough}
  |x|^2
   \geq
    \big(R|x_0-x|-x_0\cdot(x_0-x)\big)\qquad \forall x\in B.
 \ee
Some further simple computations implies that $(\ref{eq_enough})$ is equivalent to
 \be\nonumber
  (x_0-x)\cdot x\leq R^2-R|x_0-x|\qquad \forall x\in B.
 \ee
This is true since
 \bea\nonumber
  2(x_0-x)\cdot x
   =
    R^2-|x|^2-|x_0-x|^2
     \leq
      R^2-|x_0-x|^2
       \leq
        2(R^2-R|x_0-x|),
 \eea
where in the last inequality we have used $\alpha^2-\beta^2\leq2(\alpha^2-\alpha\beta)$ for all $\alpha,\beta\in\R.$

2. If the origin is an isolated point of $\partial \Gw$. Then the classical Hardy inequality near $0$ and Theorem~\ref{thm_AAP} imply that inequality \eqref{eq_diff_ineq} cannot hold.

3. It would be interesting to characterize the domains for which \eqref{eq_diff_ineq} hold true.
 }
 \end{remark}


\begin{center}{\bf Acknowledgments} \end{center}
The authors wish to thank Professor Alano Ancona for valuable discussions, and Prof. Achilles Tertikas for having pointed out the results of \cite{FTT} to them. The authors acknowledge the support of the Israel Science Foundation (grants No. 963/11) founded by the Israel Academy of Sciences and Humanities. B.~D. is supported in part by a Technion fellowship. B.~D. was also supported in part by the Natural Sciences and Engineering Research council of Canada (NSERC). G.~P. is supported in part at the Technion by a Fine Fellowship. Finally, Y.~P. would like to thank Alessia Kogoj and Sergio Polidoro and the Bruno Pini Mathematical Analysis Seminar for the kind hospitality during his visit at Universit\'{e} di Bologna and Universit\`{a} di Modena e Reggio.
\bibliographystyle{alpha}

\end{document}